\documentclass[11pt]{article}
\usepackage[latin1]{inputenc}
\usepackage{amsmath,amsthm,amssymb}

\newcommand\iLR{iL^2_r}
\newcommand\LR{L_r^2}

\newcommand\Lc{L^2_c}
\newcommand{\id}{\mathop{\rm Id}}
\newcommand{\spec}{\mathop{\rm spec}\nolimits}
\newcommand{\dom}{\mathop{\rm dom}\nolimits}
\newcommand{\im}{\mathop{\rm Im}\nolimits}
\newcommand{\re}{\mathop{\rm Re}\nolimits}
\newcommand{\Iso}{\mathop{\rm Iso}\nolimits}

\title{Generic non-selfadjoint Zakharov-Shabat operators}
\author{T. Kappeler \thanks{Supported in part by the Swiss National
Science Foundation} , P. Lohrmann \thanks{Supported in part by the Swiss National Science Foundation and the
European Research Council under FP7 \textquotedblleft New connections between dynamical systems and Hamiltonian
PDE with small divisor phenomena"} , P. Topalov \thanks{Supported in part by NSF grant DMS-0901443}}

\theoremstyle{plain}
\newtheorem{pr}{Proposition}[section]
\newtheorem{Lm}[pr]{Lemma}%[section]
\newtheorem{thm}[pr]{Theorem}%[section]

\newtheorem{cor}[pr]{Corollary}%[section]
\newtheorem{rem}[pr]{Remark}
\newtheorem{Def}{Definition}

\theoremstyle{remark}

\numberwithin{equation}{section}

\begin{document}

\maketitle

\begin{abstract}
In this paper we develop tools to study families of non-selfadjoint operators $L(\varphi ), \varphi \in P$,
characterized by the property that the spectrum of $L(\varphi )$ is (partially) simple. As a case study we
consider the Zakharov-Shabat operators $L(\varphi )$ appearing in the Lax pair of the focusing NLS on the circle.
The main result says that the set of potentials $\varphi $ of Sobolev class $H^N, N \geq 0$, so that all small
eigenvalues of $L(\varphi)$ are simple, is path connected and dense.
\end{abstract}

\section{Introduction}
In this paper we develop tools to study families of non-sefadjoint operators
$L(\varphi)$, depending on a parameter $\varphi$. To fix ideas assume that the
parameter space $P$ is a subset of some real Hilbert space and for any $\varphi\in P$,
$L(\varphi)$ has a discrete spectrum. Ideally, the spectrum $\spec L(\varphi)$
is simple for any $\varphi \in P$, i.e. any eigenvalue of $L(\varphi )$ has
algebraic multiplicity one, and $\spec L(\varphi)$ can then be
represented, under appropriate regularity assumptions on the parameter dependence of $L(\varphi)$,
$\varphi\in P$, by a family of eigenvalues $(\lambda _j(\varphi ))_{j \in J}$ with
$\lambda _j : P \rightarrow {\mathbb C}$ being real-analytic for any $j \in J$.
However, typically, such a situation does not hold and one is interested in
tools to estimate the size of the subset
\[ 
P' = \{ \varphi \in P\,|\,\spec L(\varphi ) \mbox{ simple} \}.
\]
In particular, it is of interest to know if $P'$ is open, dense, or
connected. As an illustration we recall the classical theorem of Neumann and Wigner \cite{NW}
saying that within the space $P$ of all real symmetric $n \times n$ matrices,
$n \geq 2$, the ones with multiple eigenvalues form an algebraic variety of codimensions two.
In particular, the set $P'$ of all real symmetric $n \times n$ matrices with simple spectrum is path-wise connected and 
dense. See also \cite{Ar1,Ar2}, \cite{La}, \cite{FRS}, \cite{BK} for related results.

\medskip

As a case study we consider in this paper the Zakharov-Shabat operators (ZS)
appearing in the Lax pair of the defocussing nonlinear Schr\"odinger equation
(dNLS) and the focusing one (fNLS). These operators are differential operators of
first order of the form $(x \in {\mathbb R}, \partial _x = \partial / \partial x)$
    \begin{equation}
	\label{ZSoperator} L(\varphi ) = i \begin{pmatrix} 1 & 0 \\ 0 &-1 \end{pmatrix} 
	\partial _x + \begin{pmatrix} 0 & \varphi _1 \\ \varphi _2 & 0 \end{pmatrix}
    \end{equation}
where the potential $\varphi = (\varphi _1, \varphi _2)$ is in $L^2_c = L^2
({\mathbb T}, {\mathbb C}) \times L^2 ({\mathbb T}, {\mathbb C})$ and ${\mathbb T}
= {\mathbb R} / {\mathbb Z}$. In the case of dNLS, $\varphi $ is in the subspace
$L^2_r$ whereas in the case of fNLS, $\varphi $ is in $\iLR$. Here $\LR \subseteq
L^2_c$ denotes the real subspace
   \[ L^2_r = \{ \varphi = (\varphi _1, \varphi _2) \in L^2_c\,|\,\varphi _2 =
      \overline \varphi _1 \} .
   \]
For $\varphi \in L^2_c$ arbitrary we denote by $\spec \nolimits_p L(\varphi )$ the
spectrum of the operator $L = L(\varphi )$ with domain
   \[\dom \nolimits_p L(\varphi) = \{ F \in H^1_{loc} \times H^1_{loc} \,|\,
      F(1) = \pm F(0) \} .
   \]
As $L(\varphi )$ has a compact resolvent $\spec \nolimits_p L(\varphi )$ is discrete
and each of its eigenvalues has finite algebraic multiplicity. It is referred to as the periodic spectrum of
$L(\varphi )$ and its eigenvalues as periodic eigenvalues of $L(\varphi )$ -- or
by a slight abuse of terminology as periodic eigenvalues of $\varphi $. First let
us state the following rough estimate of the periodic eigenvalues of $L(\varphi )$ -- cf. Lemma 1 in \cite{LiML}
as well as Theorem 4.1 and Theorem 4.2 in \cite{Mit}. For the convenience of the reader it is proved in Section 2.

\begin{Lm}
\label{countinglemma} [Counting Lemma] For each potential in $L^2_c$ there exist a
neighborhood $W \subseteq L^2_c$ and an integer $R \in {\mathbb Z}_{\ge 0}$ so that
for any $\varphi \in W$, when counted with their algebraic multiplicities, 
$L(\varphi )$ has two periodic eigenvalues in each disk
   \begin{equation}
   \label{roots1.1bis} D_n = \{ \lambda \in {\mathbb C}\,|\,|\lambda - n\pi | < \pi /4 \}
   \end{equation}
with $|n| > R$ and $4R + 2$ eigenvalues in the disk
   \begin{equation}
   \label{roots1.1ter} B_R = \{ \lambda \in {\mathbb C}\,|\,|\lambda | < R\pi + \pi / 4\}.
   \end{equation}
There are no other periodic eigenvalues.
\end{Lm}

The Counting Lemma shows that given a potential $\varphi$ in $L^2_c,$ for any $|n| > R$ with $R$
sufficiently large, the periodic eigenvalues of $L(\varphi )$ come in pairs, located in the
disjoint disks $D_n$. In case they are equal, one gets an eigenvalue of
geometric and algebraic multiplicity two (cf. Section 2). For $\varphi $ in $L^2_r$
or $\iLR$ one can say more. Let us first consider the case $\varphi \in L^2_r$. Then
$L(\varphi )$ is self-adjoint and hence $\spec \nolimits_p L(\varphi )$ real. It is
well known that when listed with their algebraic multiplicities, the periodic
eigenvalues are given by two doubly infinite real sequences, $(\lambda ^+_n)_{n
\in {\mathbb Z}}$ and $(\lambda ^-_n)_{n \in {\mathbb Z}}$ satisfying $\lambda ^\pm
_n = n\pi + \ell ^2_n$ and
   \[ \cdots < \lambda ^-_n \leq \lambda ^+_n < \lambda ^-_{n + 1} \leq \lambda ^+
      _{n + 1} < \cdots .
   \]
-- see e.g. \cite{GKP2} for a proof. In particular, $L(\varphi )$ has a multiple
eigenvalue iff there exists $n \in \mathbb Z$ with $\lambda ^-_n = \lambda ^+_n$.
The set $Z_n$ of potentials in $L^2_r$ with $\lambda_n^-=\lambda_n^+$ is
a real-analytic submanifold of codimension two. Hence, for any $N\in{\mathbb Z}_{\ge 0}$ the set
$L^2_r\setminus\bigcup_{|n|\le N} Z_n$ is open, dense, and connected in $L^2_r$. Furthermore,
$\bigcup_{n\in\mathbb Z} Z_n$ is dense in $L^2_r$.

%\begin{pr}
%\label{Proposition1.2} For any $n \in {\mathbb Z}$,
%   \[ {\mathcal Z}_n := \{ \varphi \in L^2_r\,|\,\lambda ^+_n(\varphi ) = \lambda ^-
%      _n(\varphi ) \}
%   \]
%is a real-analytic submanifold of codimension two. Furthermore, $\bigcup _{n \in
%{\mathbb Z}} {\mathcal Z}_n$ is dense in $L^2_r$.
%\end{pr}

For $\varphi\in\iLR$, the periodic spectrum of $L(\varphi)$ is more complicated.
If $\varphi\ne 0$, $L(\varphi)$ is not selfadjoint and hence its
periodic spectrum is not necessarily real. Moreover,  besides the asymptotic properties provided
by the Counting Lemma, the spectrum has a symmetry. For any $\lambda \in \spec \nolimits_pL(\varphi )$,
its complex conjugate $\overline \lambda $ is also a periodic eigenvalue and its
algebraic and geometric multiplicities are the same as the ones of $\lambda $ (cf. Section 2). 
In addition any real eigenvalue has geometric multiplicity two and
its algebraic multiplicity is even. No further constraints are known for the $4R + 2$
periodic eigenvalues in the disk $B_R$, given by the Counting Lemma. 
It turns out that some of the feature of $\spec_p L(\varphi)$ are still comparable
to the ones in the case where the potential is in $L^2_r$. To describe them we introduce the following notion.

\begin{Def}\label{def:standard_potential}
We say that a potential $\varphi \in \iLR$ is {\em standard}, if any real periodic eigenvalue of
$L(\varphi )$ has algebraic multiplicity two and any periodic eigenvalue in
${\mathbb C} \backslash \mathbb R $ is simple. 
\end{Def}
Denote by ${\mathcal S}_p$ the set of all standard potentials in $\iLR$. 
Due to the Counting Lemma, the property of being a standard potential involves only the $4R + 2$ eigenvalues
in $B_R$. One can show in a straightforward way that ${\mathcal S}_p$ is
open in $\iLR$ and contains the zero potential.
To state our main result we need to introduce some additional notation. For
any $N \in {\mathbb Z}_{\geq 0},$ let $H^N_c = H^N({\mathbb T}, {\mathbb C})
\times H^N({\mathbb T}, {\mathbb C})$ and $iH^N_r = H^N_c \cap iL^2_r$ where
$H^N({\mathbb T}, {\mathbb C})$ denotes the Sobolev space of functions $f :
{\mathbb T} \rightarrow {\mathbb C}$ with distributional derivatives up to order
$N$ in $L^2({\mathbb T}, {\mathbb C})$. Note that $H^0({\mathbb T},
{\mathbb C}) = L^2({\mathbb T}, {\mathbb C})$ and $H^0_c = L^2_c$.

\begin{thm}\label{THM14} 
For any $N \in {\mathbb Z}_{\geq 0}, {\mathcal S}_p \cap i H^N_r$ is path-wise connected.
\end{thm}
\begin{rem}
Concerning the proof of this theorem let us first point out that in contrast to papers such as \cite{Ar1},
the Hilbert spaces $i H^N_r$, $N\ge 0$, considered in Theorem \ref{THM14} are real. Therefore one can {\em not}
apply the standard arguments used to prove that the complement of a proper algebraic variety in a complex Hilbert space
is path-wise connected.
\end{rem}
We begin by analyzing potentials with a multiple eigenvalue
$\lambda $. It turns out that the case where the geometric multiplicity of
$\lambda $ is equal to $1$ and the one where it is $2$ have to be treated
differently. In Section 3 we show by general arguments that for any given $\psi\in \iLR$
with a periodic eigenvalue $\lambda _\psi $ of $L(\psi )$ of geometric
multiplicity one and algebraic multiplicity $m_p \geq 2$ there is a
neighborhood $W$ of $\psi \in \iLR$ so that the set of potentials in $W$,
having a periodic eigenvalue near $\lambda _\psi$ of algebraic
multiplicity $m_p$ and geometric multiplicity one, is contained in a real-analytic submanifold of
codimension two (Theorem~\ref{Theorem5.2}). A corresponding result is proved
for a potential $\psi $ in $\iLR$ admitting a periodic eigenvalue of geometric
multiplicity two (Theorem~\ref{Theorem5.3}). The proof of Theorem~\ref{Theorem5.2}
and Theorem~\ref{Theorem5.3} are based on a Theorem formulated in general terms
and proved in Appendix A, providing a class of functionals which can be used
to construct submanifolds with the properties stated in Theorem~\ref{Theorem5.2}
and Theorem~\ref{Theorem5.3}. In Section 2 we describe the set-up used
throughout the paper and in Appendix B we illustrate our results for the constant
potentials and show an auxiliary result needed in the proof of Theorem~\ref{Theorem5.2}.

It follows from the proof of Theorem~\ref{THM14} that for any $N \in
{\mathbb Z}_{\geq 0}, {\mathcal S}_p \cap iH^N_r$ is dense in $iH^N_r$.
However, there is a much easier way to prove this density result and it turns out
that a stronger result holds. First we need to introduce some more notation.
For $\varphi \in L^2_c$, denote by $\spec \nolimits_D L(\varphi )$ the Dirichlet
spectrum of the operator $L(\varphi )$, i.e., the spectrum of the operator
$L(\varphi )$ with the domain
\begin{equation}\label{domdir} 
\dom \nolimits_D L(\varphi ) = \{ f = (f_1, f_2) \in H^1([0,1],
               {\mathbb C})^2\,|\,f_1(0) = f_2(0), \ f_1(1) = f_2(1) \} .
\end{equation}
The Dirichlet spectrum is discrete and each eigenvalue has finite algebraic
multiplicity. Let
\begin{equation}\label{1.5} 
{\mathcal S}_D:= \{\varphi \in \iLR\,|\,\spec_D L(\varphi)\,\,\mbox{is simple}\} .
\end{equation}

\begin{thm}\label{Theorem1.4}
For any $N \in {\mathbb Z}_{\geq 0},$ $ {\mathcal S}_p \cap i H^N_r$ and
${\mathcal S}_D \cap i H^N_r$ are open and dense in $i H^N_r$.
\end{thm}

To prove the statement of Theorem~\ref{Theorem1.4} concerning density, we
locally reduce the problem to one for matrices and then use the discriminant
to conclude the theorem.

\medskip

The basis for the study of the geometry of the phase space of fNLS are the
spectral properties of $L(\varphi)$. Such an analysis was initiated in \cite{AbMa}
and later in more detail, taken up in \cite{LiML}. However, much remains to
be discovered -- see also \cite{Deift}. 
In a forthcoming paper we will use Theorem \ref{THM14} to construct action and
angle coordinates for the fNLS in a neighborhood of a standard potential $\varphi\in iL^2_r$.

%%%%%%%%%%%%%%%%%%%%%%%%%%%%%%%%%%%%%%%%%%%%%%%%%%%%%%%%%%%%%%%%%%%%%%%%%%%%%%%%%%%%%%
%%%%%%%%%%%%%%%%%%%%%%%%%%%%%%%%%%%%%%%%%%%%%%%%%%%%%%%%%%%%%%%%%%%%%%%%%%%%%%

\section{Set-up}
\label{secsetup}

In this section we introduce some more notations, recall several known results needed in the sequel
and establish some auxiliary results. We consider the ZS operator $L(\varphi)$,
defined by (\ref{ZSoperator}), for $\varphi=(\varphi_1,\varphi_2)$ in $\Lc$.
For any $\lambda \in \mathbb C$, let $M=M(x,\lambda,\varphi)$ be the fundamental $2\times 2$
matrix of the equation
\begin{equation*}
L(\varphi)M=\lambda M
\end{equation*}
satisfying the initial condition $M(0,\lambda,\varphi)=\text{Id}_{2\times 2}$,
\begin{equation*}
M= \left( \begin{array}{cc} m_{1} & m_{2}\\ m_{3} & m_{4} \end{array} \right ).
\end{equation*}
Further, we denote by $M_1, M_2$ the first, respectively second column of $M$.
The fundamental solution $M(x,\lambda,\varphi)$ is a continuous function on $\mathbb R \times
\mathbb C \times \Lc$ and for any given $x \in \mathbb R$, it is analytic in $\lambda, \varphi$ on
$\mathbb C \times \Lc$ \,--\, see e.g. Section 1 in \cite{GKP2}.
Moreover, the proof of Theorem 1.1 in \cite{GKP2} shows that the following stronger statement holds.
\begin{Lm} \label{toplemma20}
The fundamental matrix $M$ defines an analytic map
\[
M:\mathbb C \times \Lc \rightarrow C([0,2]), \quad (\lambda, \varphi) \mapsto M(\cdot,
\lambda, \varphi).
\]
\end{Lm}
\noindent For $\varphi=0$, the fundamental solution $E_\lambda(x):=M(x,\lambda, 0)$ is given by the
diagonal matrix $\text{diag}(e^{-i\lambda x}, e^{i \lambda x })$. In the sequel we denote by 
$(\cdot )^.$ the derivative with respect to $\lambda$. 

\vspace{0.2cm}

\noindent{\it Symmetry:} The ZS operator has various symmetries\,--\,see e.g. \cite{GK1}.
In this paper, the following one is used frequently.
For any function $f:\mathbb R \rightarrow \mathbb C^2$ with components $f_1,f_2$ introduce
the functions $\breve{f},\hat{f}:\mathbb R \rightarrow \mathbb C^2$, given by
\[
\breve{f}=(-\bar f_2, \bar f_1) \qquad \text{ and } \qquad \hat{f}=-(\bar f_2, \bar f_1).
\]
Note that for any $\varphi \in \Lc$, one has $\varphi= \hat \varphi$ iff $\varphi \in \iLR$
and that $\mathfrak{	I} f:= \breve{f}$ is an anti-involution, $\mathfrak{I}^2 f=-f$.

\begin{Lm} \label{lemma0}
Assume that $\varphi \in \Lc$, $\lambda \in \mathbb C$, and $f$ in $H^1_{loc}(\mathbb R,\mathbb
C^2)$ solves
$(L(\varphi)-\lambda)^nf=0$ for some $n \in \mathbb Z_{\geq 1}$. Then
\[
(L (\hat{\varphi})-\bar\lambda)^n \breve{f}=0.
\]
\end{Lm}

\begin{proof}
Introduce the matrices
\[
P=\left(\begin{array}{cc} 0 &  1 \\ 1 & 0  \end{array} \right), \quad R=\left(
\begin{array}{cc} 1 &  0 \\ 0 & -1  \end{array} \right), \quad J=\left(\begin{array}{cc} 0
&  -1 \\ 1 & 0  \end{array} \right).
\]
A direct computation shows that $PR=J$, $PR=-(PR)^{-1}$, $P^2=\id$, and $R^2=\id$.
As $L(\varphi)=i R \partial _x +\left(\begin{array}{cc} 0 &  \varphi_1 \\ \varphi_2 & 0
\end{array} \right)$ it then follows that
\[
PR(\overline{L(\varphi)-\lambda} )(PR)^{-1}=L(\hat{\varphi})-\bar \lambda
\]
and hence
\[
PR(\overline{L(\varphi)-\lambda})^n(PR)^{-1}=(L(\hat{\varphi})-\bar \lambda)^n.
\]
As $\breve{f}=PR \bar f$ one then concludes for any $f \in H^1_{loc}(\mathbb R,\mathbb C)$
satisfying $(L(\varphi)-\lambda )^nf=0$ that $(L(\hat \varphi)-\bar \lambda )^n\breve{f}=0$ as claimed.
\end{proof}

\medskip

\noindent{\it Periodic spectrum:}
By the definition of the fundamental solution $M$, any solution $f$ of the equation $L(\varphi)
f= \lambda f$ is given by
$f(x)=M(x, \lambda) f(0)$. Hence, a complex number $\lambda$ is a periodic eigenvalue of $L(\varphi)$
iff there exists a non zero solution of $L(\varphi) f= \lambda f$ with
\[
f(1)=M(1, \lambda) f(0)= \pm f(0).
\]
It means that $1$ or $-1$ is an eigenvalue of the Floquet matrix $M(1, \lambda)$. Denote
by $\Delta(\lambda)\equiv \Delta(\lambda, \varphi)$ the discriminant of $L(\varphi)$,
\[
\Delta(\lambda, \varphi):=m_1(1, \lambda, \varphi)+ m_4(1, \lambda, \varphi),
\]
i.e., the trace of the fundamental matrix $M$, evaluated at $x=1$. In view of the Wronskian
identity, $\det M(1, \lambda)=1$, it then follows that $\lambda $ is a periodic eigenvalue
of $L(\varphi)$ iff $\Delta(\lambda)=\pm2$.
For later reference we record the following

\begin{pr} \label{rootspr2.31}
For any $\varphi \in \Lc$, the periodic spectrum of $L(\varphi)$ coincides as a set with the zero
set of the function
\[
\chi_p(\lambda) \equiv \chi_p(\lambda, \varphi) =\Delta^2(\lambda, \varphi)-4.
\]
The discriminant $\Delta $ and hence the
characteristic function $\chi_p$ are analytic on $\mathbb C \times \Lc$.
\end{pr}

Actually, more is true. We will see below that for any periodic eigenvalue $\lambda_\varphi$
of $L(\varphi)$, the algebraic multiplicity of $\lambda_\varphi$ coincides with the multiplicity of
$\lambda_\varphi$ as a root of $\chi_p(\cdot,\varphi)$. Recall that the algebraic multiplicity of a periodic
eigenvalue $\lambda $ of $L(\varphi ), \varphi \in L^2_c$, equals the dimension
of the {\em root space} $R_\lambda (\varphi )$, defined as the following subspace of
$\dom \nolimits_p L(\varphi)$,
\[ R_\lambda (\varphi ) \!\!=\!\! \{ f \in \dom \nolimits_p L(\varphi ) \,|\, \exists\,n\in{\mathbb N}\;
\,\forall 1\le k\le n, L(\varphi)^k f\in\dom \nolimits_p L(\varphi),
       (\lambda - L(\varphi ))^n f = 0\}.
\]

\medskip

First we give the following rough localization of the roots of
$\chi_p$ \,--\, see Section 6 in \cite{GKP2}. Recall that the disks $D_n$ and $B_R$ have been
introduced in (\ref{roots1.1bis})
respectively (\ref{roots1.1ter}).

\begin{Lm} \label{lm22bis}
For each potential in $\Lc$ there exist a neighborhood $\mathcal{W}$ in $\Lc$ and
$R \in {\mathbb Z}_{\geq 0}$
such that for any $\varphi \in \mathcal{W}$ the entire function $\chi_p(\cdot,\varphi)$
has exactly two roots in each disk $D_n$ with $|n|>R$, and $4R+2$ roots in the disk $B_R$,
counted with their multiplicities. There are no other roots.
\end{Lm}

Lemma \ref{lm22bis} leads to the following corollary. To formulate it, denote by $\|\varphi \|_1$
the norm of $\varphi \in H^1_c$,
\[
\|\varphi \|_1:= (\|\varphi\|^2+ \|\partial_x \varphi \|^2)^{1/2}.
\]
\begin{cor} \label{corcount}
For any $\rho >0$ there exists $R\equiv R_\rho \geq 1$ so that for any $\varphi \in H_c^1$
with $\|\varphi \|_1 \leq \rho$, the entire function $\chi_p(\cdot, \varphi)$ has exactly two
roots in each disk $D_n$ with $|n|> R$ and exactly $4R+2$ roots in the disk $B_R$, counted with their
multiplicities. There are no other roots.
\end{cor}

\begin{proof}
For any $\varphi \in H_c^1$, let $\mathcal{W}_\varphi := \mathcal{W}$ and $R_\varphi := R$ be
as in the statement of Lemma \ref{lm22bis}. By Rellich's theorem, $H_c^1$ is compactly embedded
in $\Lc$. Hence there exist finitely many potentials $(\varphi^j)_{j \in J}$ in $H_c^1$
with $\|\varphi^j\|_1 \leq \rho$ so that $(\mathcal{W}_{\varphi^j})_{j \in J}$ covers the closed
ball of radius $\rho$ in $H_c^1$ centered at $0$.
Then $R\equiv R_\rho:= \max_{j \in J} R_{\varphi^j}$ has the claimed properties.
\end{proof}

We now prove that the algebraic multiplicity of a periodic eigenvalue equals its multiplicity
as a root of the characteristic function $\chi_p$. First we note that by functional calculus,
the algebraic multiplicity $m_p(\lambda ) \equiv m_p(\lambda , \varphi )$ of a periodic eigenvalue
$\lambda$ of $L(\varphi)$ with $\varphi \in L^2_c$ is equal to the dimension of the subspace of 
$\dom \nolimits_p L(\varphi)$, given by the image of the Riesz projector $\Pi_\lambda(\varphi)$,
\begin{equation*}
\Pi_\lambda(\varphi)= \frac{1}{2\pi i}\int_{\partial B(\lambda)} (z-L_p(\varphi))^{-1}\,dz,
\end{equation*}
where $L_p(\varphi)$ denotes the operator $L(\varphi)$ with domain $\dom \nolimits_p L(\varphi)$,
$B(\lambda)$ denotes the open disk centered at $\lambda$ with sufficiently small radius so
that $\overline{B(\lambda)}\cap \spec \nolimits_p L(\varphi)=\{\lambda\}$, and the circle
$\partial B(\lambda)$ is counterclockwise oriented. By Proposition \ref{rootspr2.31},
$\lambda$ is a root of $\chi_p(\cdot,\varphi)$. Denote by $m_r(\lambda)$ the multiplicity
of $\lambda$ as a root of $\chi_p(\cdot,\varphi)$.

\begin{Lm}\label{rootslm2.70} For any periodic eigenvalue $\lambda $ of $L(\varphi )$
with $\varphi \in L^2_c$, $m_r(\lambda)=m_p(\lambda)$.
\end{Lm}

\begin{proof}
First, note that a direct computation shows that the statement of the Lemma holds for
the zero potential $\varphi=0$. A simple perturbation argument involving Proposition \ref{rootspr2.31},
Lemma \ref{lm22bis}, the argument principle, and the properties of the Riesz projector (see the arguments below),
then shows that the Lemma continues to hold in an open neighborhood of zero in $L^2_c$.

Now, consider the general case. Take $\varphi \in \Lc$.
As $\{s\varphi\,|\, 0\leq s \leq 1 \}$ is compact in $\Lc$ there exist a connected open neighborhood $\mathcal{W}$
of the line segment $[0, \varphi]$ in $\Lc$ so that the integer $R \geq 1$ of Lemma \ref{lm22bis} can be chosen
independently of $\psi \in \mathcal{W}$. First consider the periodic eigenvalues in $B_R$.
For $\psi \in \mathcal{W}$ denote by $\Pi_R(\psi)$ the Riesz projector
\[
\Pi_R(\psi)= \frac{1}{2\pi i}\int_{\partial B_R}(z-L_p(\psi))^{-1}\, dz.
\]
Note that by functional calculus
\begin{equation}\label{eq:riezs_projector1}
\mathop{\rm Image} \Pi_R(\psi)=\oplus_{\lambda\in B_R\cap\spec \nolimits_p L(\psi)} R_\lambda(\psi)
\end{equation}
where $R_\lambda(\varphi)$ is the root space corresponding to $\lambda$.
Moreover, standard arguments show that ${\mathcal W}\to{\mathcal L}(L^2_c,L^2_c)$,
$\psi\mapsto \Pi_R(\psi)$, is analytic. In particular, by the general properties of the projection operators
the dimension of $\mathop{\rm Image}\Pi_R(\psi)$ is independent on $\psi\in{\mathcal W}$
(see \cite{Kato}, Chapter III, \S3). Consider the operator,
\[
A(\psi)=\frac{1}{2\pi i}\int_{\partial B_R} z (z-L_p(\psi))^{-1}\, dz\,.
\]
One easily sees that ${\mathcal W}\to{\mathcal L}(L^2_c,L^2_c)$,
$\psi\mapsto A(\psi)$, is analytic. By functional calculus 
$L_p(\psi)|_{\mathop{\rm Image}\Pi_R(\psi)}=A(\psi)|_{\mathop{\rm Image}\Pi_R(\psi)}$, and hence
\[
\det\Big(\lambda-L_p(\psi)|_{\mathop{\rm Image}\Pi_R(\psi)}\Big)=
\det\Big(\lambda-A(\psi)|_{\mathop{\rm Image}\Pi_R(\psi)}\Big)\,.
\]
Hence the polynomial
\[
Q(\lambda,\psi):=\det\Big(\lambda-L_p(\psi)|_{\mathop{\rm Image}\Pi_R(\psi)}\Big)
\]
is well defined, analytic in $\mathbb C \times \mathcal{W}$, and has leading coefficient one. 
By \eqref{eq:riezs_projector1}, the roots of $Q(\cdot,\psi)$ are precisely the periodic eigenvalues of
$L(\psi)$ in $B_R$ counted with their multiplicities. On the other hand, define
\[
P(\lambda, \psi):= \prod_{|j|\leq R}(\lambda-\lambda_j^+)(\lambda-\lambda_j^-).
\]
Note that $P(\lambda, \psi)$ is a polynomial in $\lambda$  of degree $4R+2$ with
leading coefficient $1$. 
By the argument principle and the last statement of Proposition \ref{rootspr2.31}, 
$P(\lambda,\psi)$ is analytic in $\mathbb C \times \mathcal{W}$.
Hence, the coefficients of $Q(\cdot,\psi)$ and $P(\cdot,\psi)$ are analytic on $\mathcal W$.
As $Q(\cdot,\psi)=P(\cdot,\psi)$ in an open neighborhood of zero in $L^2_c$ we get by 
analyticity that
\[
Q(\cdot,\psi)=P(\cdot,\psi)
\]
for any $\psi\in{\mathcal W}$. In particular, the Lemma holds also for any 
$\lambda\in B_R\cap\spec \nolimits_p L(\psi)$.
The same argument shows that the statement of the Lemma holds also for any
$\lambda\in D_n\cap\spec \nolimits_p L(\psi)$, $|n|>R$.
\end{proof}

Now we are ready to prove Lemma \ref{countinglemma} stated in the introduction.

\begin{proof}[Proof of Lemma \ref{countinglemma}]
By Lemma \ref{rootslm2.70}, for any $\varphi \in \Lc$, the roots of $\chi_p(\cdot,\varphi)$
coincide with the eigenvalues of $L_p(\varphi)$, together with the corresponding multiplicities.
Lemma \ref{countinglemma} thus follows from Lemma \ref{lm22bis}.
\end{proof}

For potentials $\varphi$ in $i L^2_r$, the results discussed so far lead to a convenient description
of the periodic spectrum of $L(\varphi)$. To state it we introduce the following order of $\mathbb C$.
We say that two complex numbers $a,b$ are {\em lexicographically ordered}, $a \preccurlyeq b$,
if $[\text{Re}(a) < \text{Re}(b)]$ or
$[ \text{Re}(a)  = \text{Re}(b) \text{ and } \text{Im}(a) \leq \text{Im}(b)]$.

\begin{pr} \label{pr2}
For any $\varphi\in iL^2_c$, any real periodic eigenvalue of $L(\varphi)$
has geometric multiplicity two and even algebraic multiplicity. For any periodic eigenvalue
$\lambda$ of $L(\varphi)$ in $\mathbb C \setminus \mathbb R$, its complex conjugate $\bar
\lambda$  is also a periodic eigenvalue of $L(\varphi)$ and has the same algebraic and
geometric  multiplicity as $\lambda$. It then follows that the periodic eigenvalues of
$L(\varphi)$, when counted with their algebraic multiplicities, are given by two doubly
infinite sequences $(\lambda_n ^+)_{ n \in \mathbb Z}$ and $(\lambda_n ^-)_{ n \in \mathbb Z}$
where $\lambda _n^-=\overline{\lambda_n^+}$ and $\im \lambda_n^+\geq 0$ for any $n \in \mathbb Z$
so that $(\lambda_n ^+)_{ n \in \mathbb Z}$ is lexicographically ordered.
\end{pr}
\begin{proof} 
It follows from Lemma \ref{lemma0}
that for any $\lambda \in \spec_p L(\varphi)$, its complex conjugate $\bar \lambda$
is in $\spec_p L(\varphi)$ as well and that $\lambda$ and $\bar \lambda$ have the same geometric
and the same algebraic multiplicities. In addition, it follows from Lemma \ref{lemma0}
that the geometric multiplicity of each real periodic eigenvalue is two. 
It remains to show that any real periodic eigenvalue of $L(\varphi)$ has
even algebraic multiplicity. Denote by $R_\lambda(\varphi)$ the root space of $L(\varphi)-
\lambda$. By Lemma \ref{lemma0}, $\mathfrak{I}f= \breve{f}$ is a $\mathbb R$-linear
anti-involution, leaving the finite dimensional vector space invariant.
Hence $\mathfrak{I}$ defines a complex structure on $R_\lambda(\varphi)$ and hence
$\dim_{\mathbb R} R_\lambda(\varphi)$ is even. This means that the algebraic
multiplicity of $\lambda$ is even.
\end{proof}

\medskip

Finally, we state the following well-known result on the asymptotics of the roots of $\chi_p$
 -- see e.g. Section 6 in \cite{GKP2}.

\begin{pr}
\label{pr1}
For any $\varphi\in \Lc$, the set of roots of $\chi_p(\cdot,\varphi)$, listed with
multiplicities, consists of a sequence of pairs $\lambda_n^-(\varphi),\lambda_n^+(\varphi)$,
$n \in {\mathbb Z}$, of complex numbers satisfying
\begin{equation*}
\lambda_n^\pm(\varphi)=n\pi +\ell^2_n
\end{equation*}
locally uniformly in $\varphi$, i.e., the sequences $(\lambda_n^\pm(\varphi)-n\pi)_{n \in
\mathbb Z}$ are locally bounded in $\ell^2(\mathbb Z, \mathbb C)$.
\end{pr}

\medskip

\noindent{\it Discriminant:}
Denote by $\dot{\Delta} $ the partial derivative of the discriminant $\Delta(\lambda,\varphi)$
with respect to $\lambda$. Then $\dot \Delta(\lambda, \varphi)$ is analytic on $\mathbb C \times \Lc$ as well. 
The following properties of $\Delta$ and $\dot \Delta $ are well known\,--\,
see e.g. Section 6 in \cite{GKP2} as well as Proposition \ref{pr2} above.
To state them, introduce
\begin{equation*}
\pi_n:=n \pi \text{  for }n \in \mathbb Z \setminus \left\{ 0 \right\}\quad \text{  and } \quad
\pi_0:=1.
\end{equation*}

\begin{pr}\label{pr4}
Let $\varphi$ be an arbitrary element in $\Lc$.

\noindent (i) The function $ \lambda \mapsto \Delta^2(\lambda,\varphi)-4 $ is entire and admits
the product representation
\begin{equation*} \Delta^2(\lambda,\varphi)-4=-4 \prod_{n\in
\mathbb Z}\frac{\left(\lambda_n^+(\varphi)-\lambda\right)\left(\lambda_n^-(\varphi)-\lambda
\right)}{\pi_n^2}.
\end{equation*}
(ii) The function $\lambda \mapsto \dot{\Delta}(\lambda, \varphi) $ is entire and has
countably many roots. They can be listed when counted with their order in such a way
that they are lexicographically ordered and satisfy the asymptotic estimates
\begin{equation*}
\dot{\lambda}_n=n\pi+\ell^2_n,
\end{equation*}
locally uniformly in $\varphi$. In addition, $\dot{\Delta}(\lambda,\varphi)$ admits
the product representation
\begin{equation*}
\dot{\Delta}(\lambda,\varphi)=2\prod_{n \in \mathbb Z}\frac{\dot{\lambda}_n -\lambda}{\pi_n}.
\end{equation*}
(iii) For any $\varphi \in \iLR$ and $\lambda \in \mathbb C$,
\[
\Delta (\bar \lambda,\varphi)= \bar{\Delta} (\lambda , \varphi) \quad \text{and}\quad
\bar{\dot{\Delta}}(\lambda,\varphi)= \dot{\Delta}( \bar \lambda,\varphi).
\]
In particular, the zero set of $\dot \Delta(\cdot, \varphi)$ is invariant under complex
conjugation. In view of the asymptotics stated  in (ii), for $n$ sufficiently large, $\dot
\lambda_n$ is real.
\end{pr}

\medskip

\noindent{\it Dirichlet spectrum:}
Recall from the introduction that for $\varphi \in \Lc$ we denote by $\spec_D L(\varphi)$
the Dirichlet spectrum of the operator $L(\varphi)$, i.e., the spectrum of the
operator $L(\varphi)$ considered with domain (\ref{domdir}).
As $L(\varphi)$, when viewed as an operator with domain $\dom_{D}(L)$ has compact
resolvent the Dirichlet spectrum is discrete. For any $\lambda \in \mathbb C$ and $\varphi \in \Lc$,
denote
\[
\grave{M}:=\left(\begin{array}{cc} \grave{m}_1 & \grave{m}_2 \\ \grave{m}_3 & \grave{m}_4
\end{array}  \right)=M(1, \lambda, \varphi)\,.
\]
Similarly as in the periodic case, one can show that the operator $L(\varphi)$ with domain
$\text{dom}_D (L)$ admits the entire function
\begin{equation*}
\chi_D(\lambda, \varphi):=  \frac{\grave{m}_4+\grave{m}_3 -\grave{m}_2-\grave{m}_1}{2i}
\end{equation*}
as a characteristic function and that the following results hold. 

\begin{Lm}\label{dircountinglemma}
For an arbitrary potential in $\Lc$ there exist a neighborhood $\mathcal{W}$ in $\Lc$ and an
integer $R \geq 1$ so that when counted with their algebraic multiplicity, for any $\varphi
\in \mathcal{W}$, there is exactly one Dirichlet eigenvalue in each disk
\[
D_n:=\left\{\lambda \in \mathbb C\,|\,|\lambda - n \pi| < \pi/ 4 \right\} \quad |n| > R,
\]
and there are exactly $2R+1$ Dirichlet eigenvalues in the disk
\[
B_R:=\left\{\lambda \in \mathbb C\,|\,|\lambda| < R\pi + \pi/4\right\}.
\]
There are no other Dirichlet eigenvalues.
\end{Lm}

\begin{pr} \label{pr3}
(i) For any $\varphi\in\Lc$, the Dirichlet eigenvalues $(\mu_n(\varphi))_{n\in \mathbb Z}$ of
$L(\varphi)$ can be listed with their algebraic multiplicities in such a way that they are
lexicographically ordered and satisfy the asymptotic estimates
\begin{equation*}
\mu_n(\varphi)=n\pi+\ell^2_n,
\end{equation*}
locally uniformly in $\varphi$. Moreover, $\chi_D(\lambda, \varphi)$ admits the product
representation
\[
\chi_D(\lambda, \varphi)=- \prod_{n \in \mathbb Z}\frac{\mu_n - \lambda}{\pi_n}.
\]
(ii) For $\varphi\in \LR$, the Dirichlet eigenvalues are real and for any $n \in \mathbb Z$
\begin{equation*}
\lambda_n^-(\varphi) \leq \mu_n(\varphi) \leq  \lambda_n^+(\varphi).
\end{equation*}
\end{pr}

By Lemma \ref{dircountinglemma} for $|n|$ sufficiently large, the Dirichlet eigenvalue $\mu_n$ is simple.
Moreover one has
\begin{Lm}\label{perdir}
(i) If for a given potential $\varphi \in \Lc$, $\lambda$ is a periodic eigenvalue of
$L(\varphi)$ of geometric multiplicity $2$, then $\lambda$ is a Dirichlet eigenvalue of
$L(\varphi)$.
(ii) If for a given potential $\varphi \in \iLR$, $\lambda $ is a real periodic eigenvalue
of $L(\varphi)$ then it is of geometric multiplicity two and hence also a Dirichlet
eigenvalue of $L(\varphi)$.
\end{Lm}

\begin{proof}
(i) If $\lambda$ is a periodic eigenvalue of $L(\varphi)$ of geometric multiplicity two,
then $M_1, M_2$ and hence $M_1+M_2$ satisfy periodic or anti-periodic boundary conditions. As
\begin{equation*}
m_1(0)+m_2(0)=1=m_3(0)+m_4(0)
\end{equation*}
it then follows that $\lambda$ is a Dirichlet eigenvalue.
(ii) follows from (i) and Proposition \ref{pr2}.
\end{proof}

\medskip

\noindent{\em $L^2$-gradients:}
Let $F : {\mathcal V}\to{\mathbb C}$ be an analytic function on an open set ${\mathcal V}$ in $L^2_c$.
The $L^2$-gradient $\partial F$ of $F$ at $\psi\in{\mathcal V}$
is an element in $L^2_c$ such that for any $h\in L^2_c$
\[
d_\psi F(h)=\langle\partial F,h\rangle_r
\]
where $d_\psi F$ denotes the differential of $F$ at $\psi$ and
\[
\langle\partial F,h\rangle_r:=\int_0^1\big((\partial_1 F)(x) h_1(x)+(\partial_2 F)(x) h_2(x)\big)\,dx\,.
\]
Let $\lambda_\varphi$ be a periodic eigenvalue of $L(\varphi)$, $\varphi\in L^2_c$, of
geometric multiplicity one. Then $\grave{M}(\lambda_\varphi,\varphi)\ne\pm\id_{2\times 2}$, and hence
$\grave{m}_2(\lambda_\varphi,\varphi)$ or $\grave{m}_3(\lambda_\varphi,\varphi)$ is not equal to
zero. The proof of the following lemma can be found e.g. in Section 4 in \cite{GKP2}
(cf. Lemma 2 in \cite{LiML}). To state it introduce the $\ast$ product,
$(f_1,f_2) \ast (g_1, g_2) := (f_2 g_2, f_1 g_1)$.
 
\begin{Lm}\label{lem:floquet}
Under the conditions listed above and if in addition $\grave{m}_2(\lambda_\varphi,\varphi)\ne 0$ 
one has 
\[
i \partial\Delta=\grave{m}_2 f\ast f
\]
where $f$ is the eigenfunction of $\lambda_\varphi$ normalized so that
\[
f(x) = M(x,\lambda_\varphi,\varphi) \binom{1}{\zeta}\quad\mbox{with}\quad
\zeta :=(\xi-\grave{m}_1)/\grave{m}_2\,,
\]
where $\xi\in\{\pm 1\}$ is the eigenvalue of $\grave{M}(\lambda_\varphi,\varphi)$. Similarly,
if $\grave{m}_3(\lambda_\varphi,\varphi)\ne 0$, then at $(\lambda_\varphi,\varphi)$
\[
i \partial\Delta=-\grave{m}_3 f\ast f
\]
where
\[
f(x) = M(x,\lambda_\varphi,\varphi ) \binom{\zeta}{1}\quad\mbox{with}\quad
\zeta :=(\xi-\grave{m}_4)/\grave{m}_3\,.
\]
\end{Lm}

Actually, the formulas for $\partial\Delta$ above can be obtained from the following formula for
$\partial\grave{M}$ (see Section 3 in \cite{GKP2}).
\begin{Lm}\label{lem:partial_M}
The $L^2$-gradient of the Floquet matrix $\grave M \equiv M(1, \lambda , \varphi )$ is given by
\begin{equation}
i\partial \grave M =
\begin{pmatrix} -\grave m_1 M_1 \ast M_2 + \grave m_2 M_1 \ast M_1
                                &-\grave m_1 M_2 \ast M_2 + \grave m_2 M_1 \ast M_2 \\
                                -\grave m_3 M_1 \ast M_2 + \grave m_4 M_1 \ast M_1
                                &-\grave m_3 M_2 \ast M_2 + \grave m_4 M_1 \ast M_2
\end{pmatrix}
\end{equation}
where $M_1$ and $M_2$ are the two column vectors of $M$ and the elements of the matrix
in parentheses are column vectors.
\end{Lm}

%%%%%%%%%%%%%%%%%%%%%%%%%%%%%%%%%%%%%%%%%%%%%%%%%%%%%%%%%%%%%%%%%%%%%%%%%%%%%%%%%%%%%%%%%%%%%%
%%%%%%%%%%%%%%%%%%%%%%%%%%%%%%%%%%%%%%%%%%%%%%%%%%%%%%%%%%%%%%%%%%%%%%%%%%%%%%

\section{Proof of Theorem \ref{THM14}}
\label{section5}

The aim of this section is to prove Theorem \ref{THM14}
saying that $\mathcal{S}_p \cap iH_r^N$ is path connected for any $N \in \mathbb Z_{\geq 0}$.
First we need to analyze multiple periodic eigenvalues of $L(\varphi)$ locally in $\iLR$.
Recall that the characteristic functions for the Dirichlet and the periodic spectrum of $L(\varphi)$,
$\varphi \in \Lc$, denoted by $\chi_D$ and $\chi_p$ respectively, are given by
\begin{equation*}
2i\chi_D(\lambda, \varphi)=\left. (\grave{m}_4+\grave{m}_3-\grave{m}_2-\grave{m}_1)
\right|_{\lambda, \varphi}
\quad\mbox{and}\quad \chi_p(\lambda, \varphi)=\left. ((\grave{m}_1+\grave{m}_4)^2-4) 
\right|_{\lambda,\varphi}.
\end{equation*}
Assume that $\lambda  \in \mathbb C$ is a periodic eigenvalue of $L(\varphi)$ of geometric multiplicity two, 
\[
m_g(\lambda , \varphi ) = 2\,.\footnote{In what follows, $m_g(\lambda)\equiv m_g(\lambda,\varphi)$
denotes the geometric multiplicity of $\lambda\in\mathbb C$ as a periodic eigenvalue of $L(\varphi)$.}
\]
By Lemma \ref{perdir} it then follows that $\lambda$ is at the same time a Dirichlet eigenvalue, i.e.,
\[
\chi_D(\lambda)=0,\quad \chi_p(\lambda)=0,\quad \text{and}
\quad \partial_\lambda \chi_p(\lambda)=0\,.
\]
One can easily see that the following more general statement holds.

\begin{Lm} \label{neulm151}
Let $\varphi \in \Lc$ and let $\lambda$  be a periodic eigenvalue of $L(\varphi)$. Then
$m_g(\lambda , \varphi ) = 2$ iff $\grave{M}(\lambda) \equiv M(1, \lambda)$ is
diagonalizable or, equivalently, $\grave{M}( \lambda) \in \left\{\pm \id_{2\times 2}\right\}$.
\end{Lm}

A periodic eigenvalue of geometric multiplicity two, $m_g(\lambda,\varphi)=2$, is said to be
{\it non-degenerate} if the algebraic multiplicity of $\lambda$, when viewed as a periodic eigenvalue
of $L(\varphi)$, is two, $m_p(\lambda , \varphi ) = 2$, and {\it degenerate} otherwise. 
Note that for the zero potential, any periodic eigenvalue is of geometric multiplicity two and
non-degenerate. 
More generally, by Lemma \ref{perdir}(ii) any real periodic eigenvalue of $L(\varphi)$ with
$\varphi \in \iLR$ is of geometric multiplicity two. It might be degenerate -- see
Corollary \ref{genericcor45}(iii) in Appendix B.
Furthermore note that a non-degenerate periodic eigenvalue of $L(\varphi)$, $\varphi \in\iLR$,
of geometric multiplicity two is not necessarily a simple Dirichlet eigenvalue.
Indeed, by Corollary \ref{genericcor45}, for the constant potential $\varphi_a=(a, -\bar a)$, $a \in \mathbb C$,
and $n \in \mathbb Z_{\ge 1}$ with $0< n \pi < |a|$, the points $\pm i \sqrt[+]{|a|^2 -n^2 \pi^2}$
are non-degenerate periodic eigenvalues of geometric multiplicity two. As $i \im(a)$ is a Dirichlet eigenvalue of
$L(\varphi_a)$, the phase of $a$ can be chosen so that $\im(a)$ equals
$\sqrt[+]{|a|^2 -n^2 \pi^2}$  -- see Corollary \ref{genericcor47}. 
For such an $a$, $i\sqrt[+]{|a|^2 -n^2 \pi^2}$ is a
Dirichlet eigenvalue of algebraic multiplicity two.

The first result concerns potentials $\psi \in \iLR$ with the property that $L(\psi )$
admits a periodic eigenvalue $\lambda _\psi $ with $m_p(\lambda_\psi ) \geq 2$ and 
$m_g(\lambda _\psi ) = 1$. In this case it is convenient to distinguish between
a {\em periodic eigenvalue in the proper sense}, characterized by $\Delta (\lambda _\psi , \psi )= 2$ and
an {\em anti-periodic eigenvalue}, characterized by $\Delta (\lambda _\psi , \psi ) = - 2$. 
The corresponding characteristic functions are
   \[ \chi ^\pm _p(\lambda , \psi ) = \Delta (\lambda , \psi ) \mp 2 .
   \]
Note that $\chi _p(\lambda , \psi ) = \chi ^+_p(\lambda , \psi ) \chi ^-_p(\lambda ,
\psi )$. Finally denote by $D^\varepsilon (\lambda _\psi ) \subseteq {\mathbb C}$ the
open disk of radius $\varepsilon > 0$ centered at $\lambda_\psi $.

\begin{thm} \label{Theorem5.2}
Assume that for $\psi \in iH^N_r$, $N\ge 0$, $\lambda_\psi$ is a periodic eigenvalue of $L(\psi)$ 
in the proper sense [alternatively, anti-periodic eigenvalue of $L(\psi)$] of algebraic multiplicity $m \geq 2$,
and geometric multiplicity one. Then for any $\varepsilon > 0$ sufficiently small there exists an open neighborhood
$\mathcal{V} \subseteq iH^N_r$ of $\psi $ such that the set
   \[ X:= \{ \varphi \in {\mathcal V}\,|\,\exists \lambda \in D^\varepsilon (\lambda _\psi )
      \mbox { with } m_p(\lambda , \varphi ) = m \mbox { and } m_g(\lambda , \varphi ) =
      1 \}
   \]
is contained in a real-analytic submanifold $Y$ of $iH^N_r$ of (real) codimension two, which
is closed in ${\mathcal V}$. In addition, ${\mathcal V}$ can be chosen so that for any
$\varphi \in {\mathcal V}$, all periodic eigenvalues of $L(\varphi )$ in $D^\varepsilon(\lambda_\psi)$ 
have geometric multiplicity one.
\end{thm}

\begin{proof} First assume that $N=0$.
As the cases where $\lambda_\psi$ is a periodic eigenvalue in the proper sense and where it
is an anti-periodic eigenvalue can be treated in the same way we concentrate on the first case only.
First we remark that due to Proposition~\ref{pr2} one has $\im(\lambda_\psi) \ne 0$. By the first
part of Theorem~\ref{TheoremA.1} applied to the characteristic function $\chi ^+_p(\lambda ,
\varphi ) = \Delta (\lambda , \varphi ) - 2$, for any $\varepsilon > 0$ sufficiently
small there exists an open neighborhood $\mathcal{V} \subseteq \iLR$ of $\psi $ so that for any $\varphi \in
\mathcal{V}, L(\varphi )$ has $m$ periodic eigenvalues $\lambda ^1(\varphi ), \ldots ,
\lambda ^m(\varphi )$, listed with their algebraic multiplicities, in the the open disk
$D^\varepsilon = D^\varepsilon (\lambda _\psi )$ and none on the boundary $\partial
D^\varepsilon $. By the characterization of the geometric multiplicity of
Lemma~\ref{neulm151}, $m_g(\lambda _\psi ) = 1$ implies that either $\grave m_2(
\lambda _\psi , \psi ) \not= 0$ or $\grave m_3(\lambda _\psi , \psi ) \not= 0$. Hence
by shrinking ${\mathcal V}$ and $\epsilon>0$ if necessary it follows that $m_g(\lambda ^k(\varphi )) = 1$
for any $1 \leq k \leq m$ and $\varphi \in {\mathcal V}$.

In order to apply Theorem~\ref{TheoremA.1}(i) we look for an analytic
function $F : {\mathbb C} \times L^2_c \rightarrow {\mathbb C}$ so that $X$ --
after shrinking $\mathcal{V}$, if necessary -- is contained in the zero set of
   \[ F_{\chi ^+_p} : \mathcal{V} \rightarrow {\mathbb C} , \quad\quad \varphi \mapsto \sum ^m
      _{j=1} F(\lambda ^j (\varphi ),\varphi)\,.
   \]
For any $q \geq 1$, take $F_q(\lambda ) = (\lambda -
\lambda _\psi )^q$.  By Theorem~\ref{TheoremA.1}(ii) applied to the pair $(F_q, \chi ^+_p)$ one concludes that
   \[ E_q : \mathcal{V} \rightarrow {\mathbb C}, \quad\quad\varphi \mapsto \sum ^m_{j=1}
      (\lambda ^j (\varphi ) - \lambda _\psi )^q
   \]
is analytic.\footnote{A function $F : \mathcal V\to\mathbb C$, $\mathcal V\subseteq iL^2_r$, is called
{\em analytic} if it is the restriction to $\mathcal V=\mathcal V_c\cap iL^2_r$ of an analytic function
$\tilde F : {\mathcal V}_c\to\mathbb C$ where ${\mathcal V}_c$ is an open set in $L^2_c$.} 
Note that for any $\varphi \in X$,
   \[ E_m(\varphi ) = m(\lambda _\varphi - \lambda _\psi )^m\quad\mbox{and}\quad
   E_1(\varphi)=m(\lambda_\varphi-\lambda_\psi)
   \]
where for $\varphi \in X$, $\lambda _\varphi $ denotes the unique periodic eigenvalue
of $L(\varphi )$ in $D^\varepsilon (\lambda _\psi )$. To obtain a functional which
vanishes on $X$ we set
   \[ G : \mathcal{V} \rightarrow {\mathbb C} , \quad\quad \varphi \mapsto m^{m - 1} E_m
      (\varphi ) - E_1(\varphi )^m .
   \]
Note that $G$ is analytic and
   \begin{equation}
   \label{5.1} G\big\arrowvert _X = 0\,.
   \end{equation}
As $\partial F_m=0$ and as $F_m(\lambda)$ has a zero of order $m$ at $\lambda=\lambda_\psi$
one concludes from Theorem~\ref{TheoremA.1}(ii) that at
$(\lambda , \varphi ) = (\lambda _\psi , \psi )$,
   \[ \partial E_m = a \partial \Delta,\;\;\;a\ne 0\,.
   \]
As $m \geq 2$ and $E_1(\psi ) = 0$ it follows that
   \[ \partial (E_1(\varphi )^m) \Big\arrowvert _{\varphi = \psi } = m E_1(\varphi )
      ^{m-1} \partial E_1 \Big\arrowvert _{\varphi =\psi } = 0
   \]
and hence at $\varphi = \psi $
   \[ \partial G = a \partial \Delta,\;\;a\ne 0.
   \]
It remains to show that near $\psi$ the zero set of $G$ is a real-analytic submanifold
of codimension two. Clearly 
\begin{equation}\label{eq:G_R,G_I}
G_R : \mathcal{V} \rightarrow {\mathbb R}, \varphi \mapsto\re G(\varphi )\quad\mbox{and}\quad
G_I : \mathcal{V} \rightarrow {\mathbb R}, \varphi \mapsto\im G(\varphi )
\end{equation}
are two real-analytic functionals. In view of the implicit function
theorem it then remains to show that the differentials $d_\psi G_R$ and $d_\psi G_I$ as elements in
${\mathcal L}(iL^2_r,\mathbb R)$ are ${\mathbb R}$-linearly independent. Recall that by assumption,
$\lambda_\psi$ has geometric multiplicity one. Then $\grave{m}_2(\lambda_\psi,\psi)$ or 
$\grave{m}_3(\lambda_\psi,\psi)$ is not equal to zero. 
Assume for simplicity that $\grave m_2(\lambda_\psi,\psi)\ne 0$. The case when $\grave m_3(\lambda_\psi,\psi)\ne 0$
is treated in the same way. It follows from Lemma \ref{lem:floquet} that at $(\lambda_\psi,\psi)$
\[ 
i \partial\Delta = \grave m_2 f \ast f
\]
where $f$ is the appropriately normalized $1$-periodic eigenfunction of $L(\psi)$ corresponding to $\lambda_\psi$.
Summarizing the computations above, one has in the case where $\grave m_2(\lambda_\psi ) \not= 0$
   \begin{equation}
   \label{5.2} \partial G = - i a \cdot \grave m_2(\lambda _\psi ) f \ast f, \;\;\;\;a\ne 0\,.
   \end{equation}
In view of Lemma~\ref{LemmaA.3} (iii) it is to show that the $\mathbb R$-linear functionals
in $i L^2_r$
   \[ \ell _R(h):= \re (\langle f \ast f, h \rangle _r) \mbox { and }
      \ell _I(h):= \im (\langle f \ast f, h \rangle _r)\,,\quad h\in i L^2_r\,,
   \]
are ${\mathbb R}$-linearly independent at $\psi$ (see the discussion before Lemma \ref{LemmaA.3} in Appendix A).
By Lemma~\ref{LemmaA.3}(iv) we know that $\ell _R$ and $\ell _I$ are ${\mathbb R}$-linearly {\it dependent}
iff there exists $c \in {\mathbb C} \backslash \{ 0 \} $ so that
   \begin{equation}
   \label{5.3} c f \ast c f + \widehat{cf \ast cf} = 0 .
   \end{equation}
Assume that \eqref{5.3} holds for some $c \not= 0$. It is convenient to introduce $g:= c f$
and $s:= \breve g=(-{\bar g}_2,{\bar g}_1)$. Then equation \eqref{5.3} reads 
\begin{equation}\label{eq:squares}
(g^2_1, g^2_2) = (s^2_1, s^2_2)\,.
\end{equation}
By Lemma~\ref{lemma0}, $s$ satisfies
   \[ L(\psi ) s = \overline \lambda _\psi s\, .
   \]
Hence,
   \begin{align}
   \label{5.5} &i g'_1 + \psi _1 g_2 = \lambda _\psi g_1 \mbox { and } i s'_1 + 
      \psi_1 s_2 = \overline \lambda _\psi s_1 \\
   \label{5.6} - &i g'_2 + \psi _2 g_1 = \lambda _\psi g_2 \mbox { and } - i s'_2 + 
      \psi_2 s_1 = \overline \lambda _\psi s_2.
   \end{align}
As $g=(g_1,g_2)\in H^1_{loc}({\mathbb R},{\mathbb C}^2)$ is a non-zero $1$-periodic solution of
$L(\psi) g=\lambda_\psi g$ we conclude that $g(x)\ne 0$ for any $x\in {\mathbb R}$.
This and the periodicity of $g$ imply that there are only four possible cases:

\medskip

{\em Case 1:} There exists a non-empty finite interval $(a,b)\subseteq{\mathbb R}$ such that $\forall x\in (a,b)$
\[
g_1(x) g_2(x)\ne 0, \quad g_1(a) g_2(a)=0, \quad\mbox{and}\quad g_1(b) g_2(b)=0;
\]

{\em Case 2:} There exists a non-empty finite interval $(a,b)\subseteq{\mathbb R}$ such that $\forall x\in (a,b)$
\[
g_1(x)=0\quad\mbox{and}\quad g_2(x)\ne 0;
\]

{\em Case 3:} There exists a non-empty finite interval $(a,b)\subseteq{\mathbb R}$ such that $\forall x\in (a,b)$
\[
g_2(x)=0\quad\mbox{and}\quad g_1(x)\ne 0;
\]

{\em Case 4:} $\forall x\in{\mathbb R}$
\[
g_1(x) g_2(x)\ne 0.
\]
First, assume that {\em Case 1} holds. It follows from \eqref{eq:squares} that on $(a,b)$,
\begin{equation}\label{eq:squares*}
g_1=\sigma_1 s_1\quad\mbox{and}\quad g_2=\sigma_2 s_2
\end{equation}
where $\sigma_1,\sigma_2\in\{\pm 1\}$.
If $\sigma_1=\sigma_2$ one obtains from \eqref{5.5} that $\im(\lambda_\psi) g_1=0$ on $(a,b)$.
As $\im(\lambda_\psi)\ne 0$ we see that $g_1=0$ on $(a,b)$. This contradicts one of the assumptions
in {\em Case 1}. Now, assume that $(\sigma_1,\sigma_2)=(1,-1)$.
Summing up the two equations in \eqref{5.5} we get that $i g_1'=\re(\lambda_\psi) g_1$ on $(a,b)$, or
$g_1(x)=\eta_1 e^{-i \re(\lambda_\psi) x}$, with constant $\eta_1\ne 0$.
Similarly, one gets from \eqref{5.6} that $g_2(x)=\eta_2 e^{i \re(\lambda_\psi) x}$, $\eta_2\ne 0$.
This implies that $g_1(x) g_2(x)=\eta_1\eta_2\ne 0$ on $(a,b)$. By continuity, $g_1(a) g_2(a)\ne 0$,
which contradicts again one of the assumptions in {\em Case 1}.
The case $(\sigma_1,\sigma_2)=(-1,1)$ is treated in the same way. Hence, {\em Case 1} does {\em not} occur.

Now, assume that {\em Case 2} holds. Then, it follows from \eqref{eq:squares} that on $(a,b)$
\[
g_1=s_1=0\quad\mbox{and}\quad g_2=\sigma s_2
\]
where $\sigma\in\{\pm 1\}$.
This together with \eqref{5.6} implies that $\im(\lambda_\psi) g_2=0$ on $(a,b)$. As $\im(\lambda_\psi)\ne 0$ we
see that $g_2=0$ on $(a,b)$. This contradicts one of the conditions in {\em Case 2}.
In the same way one treats {\em Case 3}.

Finally, consider {\em Case 4}. Arguing as in {\em Case 1} we see that \eqref{eq:squares*} holds and
the only possible cases are $(\sigma_1,\sigma_2)=(1,-1)$ and $(\sigma_1,\sigma_2)=(-1,1)$.
If $(\sigma_1,\sigma_2)=(1,-1)$ one concludes from \eqref{5.5} that
   \begin{equation}
   \label{5.7} ig'_1 = \re (\lambda _\psi )g_1 \mbox { and } \psi _1 g_2 = i
               \im (\lambda _\psi ) g_1
   \end{equation}
and from \eqref{5.6} that
   \begin{equation}
   \label{5.8} \psi _2 g_1 = i \im(\lambda _\psi )g_2 \mbox { and } -i g'_2 =
               \re (\lambda _\psi ) g_2 .
   \end{equation}
Hence
   \[ g_1(x) = \eta _1\,e^{-i \re(\lambda_\psi) x} , \ g_2(x) = \eta _2\,e^{i \re(\lambda_\psi) x}
   \]
with $\eta _1, \eta _2$ in ${\mathbb C} \backslash \{ 0 \}$. 
Solving \eqref{5.7}-\eqref{5.8} for $\psi _1, \psi _2$ one then gets
   \[ \psi _1 = i \im(\lambda_\psi)\frac{g_1}{g_2} = i \im(\lambda_\psi)
      \frac{\eta _1}{\eta _2}\,e^{-2 i \re(\lambda_\psi) x}
   \]
and
   \[ \psi _2 = i \im (\lambda _\psi ) \frac{g_2}{g_1} = i \im(\lambda_\psi)
      \frac{\eta _2}{\eta _1}\,e^{2 i \re(\lambda _\psi) x} .
   \]
As $\psi \in \iLR$ and thus $\overline \psi_1 = - \psi _2$ one has $\eta _1 /\eta _2 = e^{i\alpha }$ 
with $\alpha\in\mathbb R$, and as $\psi $ is $1$-periodic it follows that $\re (\lambda
_\psi ) = k\pi $ for some $k \in {\mathbb Z}$. Hence
   \[ \psi _1(x) = i \im (\lambda _\psi ) e^{i\alpha }e^{-2k\pi i x} \mbox { and }
      \lambda _\psi = k\pi + i \im (\lambda _\psi ) .
   \]
By Lemma~\ref{Lemma8.5}, $\lambda _\psi = k\pi + i\im(\lambda _\psi)$ has algebraic
multiplicity one. This contradicts the assumption $m_p(\lambda _\psi ) = m \geq 2$.
The case $(\sigma _1, \sigma _2) = (-1,1)$ is treated in the same way as the case
$(\sigma _1,\sigma _2) = (1,-1)$. Altogether we have shown that $\ell _R$ and $\ell _I$, and hence
the differentials $d_\psi G_R$ and $d_\psi G_I$, are ${\mathbb R}$-linearly independent.
As $G\big\arrowvert _X = 0$, the claimed statement concerning $X$ then follows from
the implicit function theorem. 

Finally, assume that $N\ge 1$. Take $\psi\in iH^N_r$ and note that
the restrictions $G_R\big|_{{\mathcal V}\cap iH^N_r}$ and $G_I\big|_{{\mathcal V}\cap iH^N_r}$
of the functionals \eqref{eq:G_R,G_I} considered above are real-analytic. Moreover, for any $h\in iH^N_r$,
\[
d_\psi (G_R\big|_{{\mathcal V}\cap iH^N_r})(h)=\langle\partial G_R,h\rangle_r\quad\mbox{and}\quad
d_\psi (G_I\big|_{{\mathcal V}\cap iH^N_r})(h)=\langle\partial G_I,h\rangle_r\,.
\]
Assume that there exist $\alpha,\beta\in\mathbb R$, $(\alpha,\beta)\ne 0$, such that for any $h\in iH^N_r$,
$\alpha\langle\partial G_R,h\rangle_r+\beta\langle\partial G_I,h\rangle_r=0$.
As $iH^N_r$ is dense in $iL^2_r$ we see by a continuity argument that the last equality holds also for
any $h\in iL^2_r$. As this contradicts to the result obtained in the case $N=0$ we get that the differentials
$d_\psi (G_R\big|_{{\mathcal V}\cap iH^N_r})$ and $d_\psi (G_I\big|_{{\mathcal V}\cap iH^N_r})$ are
$\mathbb R$-linearly independent in $\mathcal L(i H^N_r,\mathbb R)$.
Then, arguing as in the case $N=0$ we complete the proof of Theorem \ref{Theorem5.2}.
\end{proof}

The second result deals with potentials $\psi \in \iLR$ with the property
that $L(\psi )$ admits a periodic eigenvalue $\lambda _\psi $ with
$m_g(\lambda _\psi ) = 2$. In this case,
$\grave M(\lambda _\psi ) \in \{ \pm \rm{Id}_{2\times 2}\}$ and hence $\lambda _\psi$
is at the same time a Dirichlet eigenvalue. Denote by $m_D(\lambda_\psi)$
the algebraic multiplicity of $\lambda_\psi $ as Dirichlet eigenvalue.

\begin{thm}\label{Theorem5.3}
Assume that for $\psi $ in $iH^N_r$, $N\ge 0$, $\lambda_\psi$ is a periodic eigenvalue
of $L(\psi)$ in the proper sense [alternatively, anti-periodic eigenvalue of $L(\psi)$]
with $m_g(\lambda _\psi ) = 2$ and $m_D(\lambda_\psi)=m\ge 1$. 
Then for any $\varepsilon > 0$ sufficiently small there exists an open neighborhood
$\mathcal{V}\subseteq iH^N_r$ of $\psi$ such that the
set
   \[ X := \{ \varphi \in \mathcal{V}\,|\,\exists \lambda \in D^\varepsilon (\lambda
      _\psi ) \mbox{ with } m_g(\lambda , \varphi ) = 2, m_D(\lambda , \varphi ) = m \}
   \]
is contained in a real-analytic submanifold $Y$ in $i H^N_r$ of (real) codimension two,
which is closed in ${\mathcal V}$.
\end{thm}

\begin{proof} 
As the case $N\in{\mathbb Z}_{\geq 1}$ is treated in the same way as $N = 0$ -- see the proof
of Theorem \ref{Theorem5.2} above -- we concentrate on the latter case only.
Similarly, as the cases where $\lambda_\psi$ is a periodic eigenvalue in the proper sense and where
it is an anti-periodic eigenvalue can be treated in the same way we concentrate on the first case only. It turns out
that we have to distinguish between two different cases. We begin with the case where
$m = m_D(\lambda _\psi ) \geq 2$.

{\it Case 1: $m \geq 2$}. By Theorem~\ref{TheoremA.1}(i), applied to
the characteristic function $\chi _D(\lambda , \varphi ) = \frac{i}{2} (\grave m_1 +
\grave m_2 - \grave m_3 - \grave m_4)\big\arrowvert _{\lambda , \varphi }$, for any
$\varepsilon > 0$ sufficiently small there exists an open neighborhood $\mathcal{V} \subseteq {iL}^2_r$ of
$\psi $ so that for any $\varphi \in \mathcal{V}, L(\varphi )$ has $m$ Dirichlet
eigenvalues $\mu ^1(\varphi ), \ldots , \mu ^m(\varphi )$, listed with their
algebraic multiplicities, in the open disk $D^\varepsilon \equiv D^\varepsilon (\lambda_\psi)$ and
none on the boundary $\partial D^\varepsilon $. Similarly as in the proof
of Theorem~\ref{Theorem5.2}, introduce the functional
   \begin{equation}
   \label{5.9} G : \mathcal{V} \rightarrow {\mathbb C} ,\quad\quad\varphi \mapsto m^{m-1}
               E_m(\varphi ) - E_1(\varphi )^m
   \end{equation}
where here, for any $q \geq 1$,
   \[ E_q : \mathcal{V} \rightarrow {\mathbb C}, \quad\quad\varphi \mapsto \sum ^m_{j=1}
      (\mu ^j(\varphi ) - \lambda _\psi )^q .
   \]
By Theorem~\ref{TheoremA.1}(ii), applied to $F_q(\lambda ):= (\lambda - \lambda _\psi )^q$
and $\chi _D(\lambda , \varphi )$, one concludes that $E_q$ is analytic for any $q \geq 1$.
Note that for any $\varphi \in X$,
   \[ E_m(\varphi ) = m(\mu _\varphi - \lambda _\psi )^m\quad\mbox{and}\quad
         E_1(\varphi)=m(\mu_\varphi-\lambda_\psi)
   \]
where for $\varphi \in X, \mu _\varphi $ denotes the unique Dirichlet eigenvalue of
$L(\varphi )$ in $D^\varepsilon $. It then follows that 
   \[ G \big\arrowvert _X = 0\,.
   \]
As $\partial F_m=0$ and as $F_m(\lambda)$ has a zero of order $m$ at $\lambda=\lambda_\psi$ one
concludes from Theorem~\ref{TheoremA.1}(ii) and \eqref{5.9} that at
$(\lambda , \varphi ) = (\lambda _\psi , \psi )$,
   \begin{equation}
   \label{5.10}   \partial G = a \partial \chi _D,\quad a\ne 0\,.
   \end{equation}
By Lemma \ref{lem:partial_M}, the $L^2$-gradient of the Floquet matrix $\grave M \equiv
M(1, \lambda , \varphi )$ is given by
   \begin{equation}
   \label{5.12} i\partial \grave M =
                \begin{pmatrix} -\grave m_1 M_1 \ast M_2 + \grave m_2 M_1 \ast M_1
                                &-\grave m_1 M_2 \ast M_2 + \grave m_2 M_1 \ast M_2 \\
                                -\grave m_3 M_1 \ast M_2 + \grave m_4 M_1 \ast M_1
                                &-\grave m_3 M_2 \ast M_2 + \grave m_4 M_1 \ast M_2
   \end{pmatrix}
   \end{equation}
where $M_1$ and $M_2$ are the two column vectors of $M$ and the elements of the matrix
in parentheses are column vectors. Thus
   \begin{align} 2\partial \chi _D &= i \partial \grave m_1 + i \partial \grave m_2 -
                     i \partial \grave m_3 - i \partial \grave m _4 \nonumber\\
                  &= (\grave m_2 - \grave m_4) M_1 \ast  M_1 + (\grave m_3 - \grave m_1)
                     M_2 \ast M_2 \nonumber\\
                  &+ (\grave m_2 + \grave m_3 - \grave m_1 - \grave m_4) M_1 \ast M_2\,.\label{5.12'}
   \end{align}
As at $(\lambda , \varphi ) = (\lambda _\psi , \psi ), \grave M = Id_{2\times 2}$ one gets
   \begin{equation}
   \label{5.13} 2 \partial \chi _D = - M_1 \ast M_1 - M_2 \ast M_2 - 2 M_1 \ast M_2 .
   \end{equation}
By Lemma \ref{toplemma20}, $M(\cdot,\lambda_\psi,\psi)\in C([0,2])$.
In particular, it can be evaluated at $x = 0$. 
One thus obtains
   \[ 2\partial \chi _D(0, \lambda _\psi , \psi ) = - \binom{1}{1} .
   \]
In view of \eqref{5.10}, 
   \begin{equation}\label{5.15} 
       \partial G\big\arrowvert_{x=0} = - \frac{a}{2} \binom{1}{1},\quad a\ne 0\,.
   \end{equation}
In addition to $G$ we need to introduce a second functional,
denoted by $H$,
   \[ H : \mathcal{V} \rightarrow {\mathbb C},\quad\quad\varphi \mapsto \sum ^m_{j=1} \grave m_2(\mu ^j
      (\varphi ),\varphi ) .
   \]
Note that $H\big\arrowvert _X = 0$. By Lemma~\ref{toplemma20}, $\grave m_2 :
{\mathbb C} \times L^2_c \rightarrow {\mathbb C}, (\lambda , \varphi ) \mapsto
m_2(1, \lambda , \varphi )$ is analytic.
By Theorem~\ref{TheoremA.1}, applied to $(F, \chi ) =
(\grave m_2, \chi _D)$ it follows that $H$ is analytic and that at $(\lambda _\psi , \psi )$
   \[ \partial H = m\,\partial \grave m_2 + \sum ^m_{j=0} a_j \partial ^{m-j}_\lambda
      \partial \chi _D .
   \]
As $\grave m_2(\lambda _\psi , \psi ) = 0, a_0 = 0$ by Theorem~\ref{TheoremA.1} and
one gets at $(\lambda _\psi , \psi )$
   \begin{equation}
   \label{5.16} \partial H = m\,\partial \grave m_2 + \sum ^m_{j=1} a_j \partial ^{m-j}
                _\lambda \partial \chi _D .
   \end{equation}
Let us first discuss the term $m\,\partial \grave m_2$ in more detail. By \eqref{5.12}
one has
   \[ i\,\partial \grave m_2 = - \grave m_1 M_2 \ast M_2 + \grave m_2 M_1 \ast M_2 .
   \]
As $\grave M = {\rm Id}_{2\times 2}$ at
$(\lambda _\psi , \psi )$ one then gets
   \begin{equation}
   \label{5.17} i\,\partial \grave m_2 \Big\arrowvert _{x = 0} = - \binom{1}{0} \ \mbox { and }
                \ m\,\partial \grave m_2 \Big\arrowvert _{x = 0} = i m\,\binom{1}{0} .
   \end{equation}
Next, let us turn to the second term of the right hand side of formula \eqref{5.16}.
It follows from \eqref{5.12'} and Lemma \ref{toplemma20} that
\[
{\mathbb C}\to C([0,2]),\,\,\lambda\mapsto \partial\chi _D(\cdot,\lambda,\psi)
\]
is analytic.
This implies that
   \[ \partial ^k_\lambda \partial \chi _D(\cdot , \lambda , \psi ) \Big\arrowvert
      _{x = 0} = \partial ^k_\lambda \left( \partial \chi _D(0, \lambda , \psi )
      \right) .
   \]
For any $\lambda\in\mathbb C$
   \begin{align*} 2 \partial \chi _D(0,\lambda , \psi ) &= (\grave m_2 - \grave m_4)
                  \binom{0}{1} + (\grave m_3 - \grave m_1) \binom{1}{0} \\
                  &= (\grave m_3 - \grave m_1) \binom{1}{1} + (\grave m_1 + \grave m_2
                  - \grave m_3 - \grave m_4) \binom{0}{1}
   \end{align*}
or
   \begin{equation}
   \label{5.14} 2 \partial \chi _D(0, \lambda , \psi ) = (\grave m_3 - \grave m_1)
                \binom{1}{1} - 2i \chi _D(\lambda , \psi )\binom{0}{1} .
   \end{equation}   
As by assumption, $\partial ^k_\lambda \chi _D(\lambda _\psi , \psi ) = 0$ for any
$0 \leq k \leq m - 1$, it then follows from formula \eqref{5.14} that
   \[ 2 \partial ^{m-j}_\lambda \left(\partial\chi_D(0,\lambda,\psi)\right)\big|_{\lambda=\lambda_\psi} =
      \partial ^{m-j}_\lambda (\grave m_3 - \grave m_1)\big|_{\lambda=\lambda_\psi} \binom{1}{1}
   \]
for any $1 \leq j \leq m$. When combined with \eqref{5.16} and \eqref{5.17}
one has
   \begin{equation}\label{5.18} 
\partial H\big\arrowvert_{x=0} = i m \binom{1}{0} + \kappa
             \binom{1}{1}
   \end{equation}
for some $\kappa\in\mathbb C$.

Following the notation introduced in Appendix A, denote by
$\ell = \ell^G : \iLR \rightarrow {\mathbb C}$ the ${\mathbb R}$-linear functional induced
by $\partial G = (\partial _1 G, \partial _2G)$
   \[ \ell^G(h):= \langle \partial G, h \rangle _r = \int ^1_0 (\partial _1 Gh_1 +
      \partial _2G h_2)dx
   \]
and let
   \[ \ell ^G_R(h):= \re (\langle \partial G, h \rangle _r) \mbox { and }
        \ell^G_I(h):= \im (\langle \partial G, h \rangle _r) .
   \]
According to \eqref{A.4}, one has for $h\in iL^2_r$
   \[ \partial _s \Big\arrowvert _{s=0} \re G(\varphi + sh) = \ell ^G_R(h) =
      \Big\langle \frac{\partial G + \widehat{\partial G}}{2} , h \Big\rangle _r
   \]
and similarly
   \[ \partial _s \Big\arrowvert _{s=0} \im G(\varphi + sh) = \ell ^G_I(h) =
      \Big\langle \frac{\partial G - \widehat{\partial G}}{2i} , h \Big\rangle _r
   \]
where we recall that for $f = (f_1, f_2) \in L^2_c$, $\hat f$ is given by
$\hat f = - (\overline f_2, \overline f_1)$. By formula \eqref{5.15}, for $\varphi =
\psi $,
   \begin{equation}
   \label{5.20} \frac{1}{2} (\partial G + \widehat{\partial G})\Big\arrowvert _{x=0}
                = \frac{1}{4}(\bar{a} - a)\binom{1}{1} = - \frac{i}{2} \im(a) \binom{1}{1}
   \end{equation}
and
   \begin{equation}
   \label{5.21} \frac{1}{2i} (\partial G - \widehat{\partial G})\Big\arrowvert _{x=0}
                = - \frac{1}{4i}(a + \overline a)\binom{1}{1} =  \frac{i}{2} \re(a) \binom{1}{1}
   \end{equation}
where $a \not= 0$. Similarly, we define for $H$
   \[ \ell ^H_R(h) = \Big\langle \frac{\partial H + \widehat{\partial H}}{2}, h
      \Big\rangle _r \mbox { and } \ell ^H_I(h) = \Big\langle \frac{\partial H -
      \widehat{\partial H}}{2i}, h \Big\rangle _r .
   \]
By formula \eqref{5.18}, at $\varphi = \psi $,
   \begin{equation}
   \label{5.22} \frac{1}{2} \big( \partial H + \widehat{\partial H} \big)
                \Big|_{x=0} = i \big(m/2 + \im(\kappa) \big)
                \binom{1}{1}
   \end{equation}
and
   \begin{equation}
   \label{5.23} \frac{1}{2i} \big( \partial H - \widehat{\partial H} \big)
                \Big|_{x=0} = \frac{m}{2}\binom{1}{-1} - i \re(\kappa)
                \binom{1}{1}.
   \end{equation}
In view of the identities \eqref{5.20} - \eqref{5.23} introduce
   \[ F_1 : {\mathcal V} \rightarrow {\mathbb R},\quad\quad\varphi \mapsto \im H(\varphi )
   \]
and
   \[ F_2 : {\mathcal V} \rightarrow {\mathbb R},\quad\quad\varphi \mapsto
      \begin{cases} \re G(\varphi), &\im(a) \not= 0 \\ \im G(\varphi),
      &\im(a) = 0 \end{cases}\,.
   \]
As $a\ne 0 $, $\im(a) = 0$ implies that $\re(a)\ne 0$ and
hence according to \eqref{5.21}, $\frac{1}{2i} (\partial G - \widehat{\partial G})
\big\arrowvert _{x=0} = - \frac{i}{2} \re(a) \binom{1}{1} \not= 0$. Now define
   \[ Y = \{ \varphi\in {\mathcal V}\,|\,F_1(\varphi ) = 0, F_2(\varphi ) =
      0 \} .
   \]
By construction, $G\big\arrowvert _X = 0$, $H\big\arrowvert _X = 0$ and hence $X \subseteq Y$.
By \eqref{5.20}, \eqref{5.21}, and \eqref{5.23}, $\partial F_1$ and
$\partial F_2$ are ${\mathbb R}$-linearly independent at $\varphi = \psi $. By
the implicit function theorem, it then follows that after shrinking ${\mathcal V}$,
if necessary, $X$ is contained in a real-analytic
submanifold of $\iLR$ of codimension two. Hence the claimed result for $X$ is established
in {\em Case 1}.

{\it Case 2: $m=m_D(\lambda_\psi)=1$ \& $m_g(\lambda_\psi)=2$.} By Theorem~\ref{TheoremA.1}(i), applied to
the characteristic function $\chi_D$, for any $\varepsilon > 0$ sufficiently small there exists an open neighborhood
${\mathcal V} \subseteq\iLR$ of $\psi $ so that for any $\varphi \in {\mathcal V}, L(\varphi)$ has precisely
one Dirichlet eigenvalue, denoted by $\mu (\varphi )$ in the open disk
$D^\varepsilon = D^\varepsilon(\lambda _\psi)$ and none on the boundary $\partial D^\varepsilon$.
As $\mu(\varphi)$ is simple, it follows from the inverse function theorem that
the mapping $\mu : \mathcal V\to\mathbb C$ is analytic. In view of Lemma \ref{neulm151},
\begin{equation}\label{eq:subset}
X\subseteq\{\varphi\in{\mathcal V}\,|\,\grave{m}_2(\mu(\varphi),\varphi)=\grave{m}_3(\mu(\varphi),\varphi)=0\}\,.
\end{equation}
Consider the functionals,
\[
H_1 : \mathcal V\to\mathbb C,\quad\quad\varphi\mapsto\grave{m}_2(\mu(\varphi),\varphi)
\]
and
\[
H_2 : \mathcal V\to\mathbb C,\quad\quad\varphi\mapsto\grave{m}_3(\mu(\varphi),\varphi)\,.
\]
In view of Lemma \ref{toplemma20}, $H_1$ and $H_2$ are analytic, and by \eqref{eq:subset},
\[
H_1\big|_X=H_2\big|_X=0\,.
\]
Next, we will compute the $L^2$-gradients of $H_1$ and $H_2$ at $\varphi=\psi$. By the chain rule,
we have that at $\varphi=\psi$
\begin{equation}\label{eq:H_1}
\partial H_1=\partial_\lambda\grave{m}_2(\lambda_\psi,\psi)\,\partial\mu+\partial\grave{m}_2(\lambda_\psi,\psi)
\end{equation}
and 
\begin{equation}\label{eq:H_2}
\partial H_2=\partial_\lambda\grave{m}_3(\lambda_\psi,\psi)\,\partial\mu+\partial\grave{m}_3(\lambda_\psi,\psi)\,.
\end{equation}
Using the identity $\chi_D(\mu(\varphi),\varphi)=0$ for $\varphi\in\mathcal V$ and that
${\dot\chi_D}(\mu(\varphi),\varphi)\ne 0$ by the assumed simplicity of $\mu(\varphi)$
we obtain that
\[
\partial\mu=-\frac{1}{\dot\chi_D}\,\partial\chi_D
\]
where $\partial\chi_D=\partial\chi_D(\lambda_\psi,\psi)$ and   
${\dot\chi_D}={\dot\chi_D}(\lambda_\psi,\psi)$.
By \eqref{5.13},
\[ 
-2\partial\chi_D = M_1\ast M_1 + M_2 \ast M_2 + 2 M_1 \ast M_2\,,
\]
and hence,
\begin{equation}\label{eq:partial_mu*}
\partial\mu=\frac{1}{2 \dot\chi_D}\,\big(M_1\ast M_1 + M_2 \ast M_2 + 2 M_1 \ast M_2\big)\,.
\end{equation}
In particular,
\begin{equation}\label{eq:partial_mu}
\partial\mu\big|_{x=0}=\frac{1}{2 \dot\chi_D}\,\binom{1}{1}\,.
\end{equation}
By Lemma \ref{lem:partial_M},
\begin{equation}\label{eq:partial_m_2*}
\partial\grave{m}_2=i \grave{m}_1 M_2\ast M_2 - i \grave{m}_2 M_1\ast M_2
\end{equation}
and
\begin{equation}\label{eq:partial_m_3*}
\partial\grave{m}_3=i \grave{m}_3 M_1\ast M_2 - i \grave{m}_4 M_1\ast M_1\,.
\end{equation}
As at $(\lambda,\varphi)=(\lambda_\psi,\psi)$, $\grave M = Id_{2\times 2}$ we get that
\begin{equation}\label{eq:partial_m_2}
\partial\grave{m}_2\big|_{x=0}=i \binom{1}{0}
\end{equation}
and
\begin{equation}\label{eq:partial_m_3}
\partial\grave{m}_3\big|_{x=0}= -i \binom{0}{1}\,.
\end{equation}
Combining \eqref{eq:H_1}-\eqref{eq:partial_m_3} we then obtain at $\varphi=\psi$
\begin{equation}\label{eq:partial_H_1*}
\partial H_1 = \kappa_1\big(M_1\ast M_1+M_2\ast M_2+2 M_1\ast M_2\big) + i M_2\ast M_2
\end{equation}
\begin{equation}\label{eq:partial_H_2*}
\partial H_2 = \kappa_2 \big(M_1\ast M_1+M_2\ast M_2+2 M_1\ast M_2\big) - i M_1\ast M_1
\end{equation}
and
\begin{equation*}%\label{eq:partial_H_1}
\partial H_1\big|_{x=0}=\kappa_1\binom{1}{1} + i \binom{1}{0}
\end{equation*}
\begin{equation*}%\label{eq:partial_H_2}
\partial H_2\big|_{x=0}=\kappa_2\binom{1}{1} - i \binom{0}{1}
\end{equation*}
where 
\begin{equation}\label{eq:kappa_1,2}
\kappa_1:=\partial_\lambda\grave{m}_2/{2\dot\chi_D}\quad\mbox{and}\quad
\kappa_2:=\partial_\lambda\grave{m}_3/{2\dot\chi_D}\,.
\end{equation}
Arguing as in {\it Case 1} we compute
\begin{equation}\label{eq:H_1^R}
\frac{1}{2} \big(\partial H_1 + \widehat{\partial H_1}\big)\big|_{x=0}
                = i \big(1/2+\im(\kappa_1)\big)\,\binom{1}{1}
\end{equation}
\begin{equation}\label{eq:H_1^I}
\frac{1}{2 i} \big(\partial H_1 - \widehat{\partial H_1}\big)\big|_{x=0}
                = \frac{1}{2}\,\binom{1}{-1} - i \re(\kappa_1)\,\binom{1}{1}
\end{equation}
and
\begin{equation}\label{eq:H_2^R}
\frac{1}{2} \big(\partial H_2 + \widehat{\partial H_2}\big)\big|_{x=0}
                = i \big(-1/2+\im(\kappa_2)\big)\,\binom{1}{1}
\end{equation}
\begin{equation}\label{eq:H_2^I}
\frac{1}{2 i} \big(\partial H_2 - \widehat{\partial H_2}\big)\big|_{x=0}
                = \frac{1}{2}\,\binom{1}{-1} - i \re(\kappa_2)\,\binom{1}{1}\,.
\end{equation}
For any analytic function $F : \mathcal V\to{\mathbb C}$ denote for simplicity
\[
\partial_R F:= \frac{1}{2} \big(\partial F + \widehat{\partial F}\big)
\quad\mbox{and}\quad\partial_I F:= \frac{1}{2 i} \big(\partial F - \widehat{\partial F}\big)\,.
\]
We now show that
\begin{equation}\label{eq:rank}
\mathop{\rm rank}\nolimits_{\mathbb R}\{\partial_R H_1, \partial_I H_1, \partial_R H_2, \partial_I H_2\}\ge 2\,.
\end{equation}
Assume on the contrary that the rank above is one.
Then it follows from \eqref{eq:H_1^R}-\eqref{eq:H_2^I} that
\begin{equation}\label{eq:kappa}
\kappa_1=a-\frac{i}{2}\quad\mbox{and}\quad\kappa_2=a+\frac{i}{2},\quad\mbox{where}\,\,\,\,a\in\mathbb R\,.
\end{equation}
This together with \eqref{eq:partial_H_1*} and \eqref{eq:partial_H_2*} imply that at $\varphi=\psi$
\begin{equation}\label{eq:difference}
\partial H_1-\partial H_2=-2 i M_1\ast M_2\,.
\end{equation}
As the rank in \eqref{eq:rank} is assumed to be one, we get from Lemma \ref{lem:independence} below that
\[
M_1\ast M_2\equiv 0\,.
\]
It means that at $\varphi = \psi $, for any $0 \leq x \leq 1$
   \begin{equation}
   \label{5.28} m_1(x,\lambda _\psi )m_2(x,\lambda _\psi ) = 0  \mbox { and } \ m_3(x,
                \lambda _\psi ) m_4(x,\lambda _\psi ) = 0 .
   \end{equation}
Multiplying the first row of $L(\psi ) M_1 = \lambda _\psi M_1$ by $m_4$ and the second
by $m_2$ yields
   \[ im'_1m_4 = \lambda _\psi m_1m_4\quad\mbox{and}\quad -im'_3 m_2 = \lambda _\psi m_3
      m_2 .
   \]
Taking the difference of the two equations and using the Wronskian identity one then
gets
   \begin{equation}
   \label{5.29} i(m'_1 m_4 + m'_3 m_2) = \lambda _\psi .
   \end{equation}
In the same way one gets, after multiplying the first row of $L(\varphi )M_2 = \lambda
_\psi M_2$ by $m_3$ and the second by $m_1$
   \[ im'_2 m_3 = \lambda _\psi m_2 m_3\quad\mbox{and}\quad -im'_4 m_1 = 
\lambda _\psi m_4 m_1
   \]
leading to
   \begin{equation}
   \label{5.30} i(m'_2m_3 + m'_4m_1) = - \lambda _\psi .
   \end{equation}
Adding \eqref{5.29} and \eqref{5.30} one obtains
   \[ \partial _x(m_1m_4 + m_2m_3) = 0 
   \]
or, in view of the Wronskian identity,
   \[ \partial _x(m_1 m_4) = 0\,.
   \]
As $m_1 m_4\big\arrowvert _{x=0} = 1$ one therefore has
   \[ m_1(x,\lambda _\psi ) m_4(x,\lambda _\psi ) = 1 \quad \forall \ 0 \leq x \leq 1 .
   \]
This combined with \eqref{5.28} leads to
   \[ m_2(x, \lambda _\psi ) = 0 \mbox { and } m_3(x,\lambda _\psi ) = 0 \quad
      \forall \ 0 \leq x \leq 1 .
   \]
Multiplying the first row of $L(\psi )M_2 =
\lambda _\psi M_2$ by $m_1$ and using that $m_2 = 0$ yields
   \[ 0 = \psi _1 m_4 m_1 = \psi _1 .
   \]
As $\psi$ is in $\iLR$ one has $\psi_2 = - \overline \psi _1$ and hence
\[
\psi = 0\,. 
\]
A simple computation (cf. Lemma \ref{genericlm44}) shows that
\[
\grave{M}(\lambda,\psi)\Big|_{\psi=0}=\left(
\begin{array}{cc}
e^{-i \lambda}&0\\
0&e^{i \lambda}
\end{array}
\right)\,.
\]
Hence, 
\[
\grave{m}_2(\lambda,0)=\grave{m}_3(\lambda,0)\equiv 0\quad{and}\quad\chi_D(\lambda,0)=\sin\lambda
\]
and by \eqref{eq:kappa_1,2}
\[
\kappa_1=\kappa_2=0.
\]
This contradicts \eqref{eq:kappa}. Therefore, \eqref{eq:rank} holds.
By the implicit function theorem it then follows that, after shrinking $\mathcal V$ if necessary,
$X$ is contained in a real-analytic submanifold in $iL^2_r$ of codimension two.
\end{proof}

\begin{Lm}\label{lem:independence}
If $M_1\ast M_2\not\equiv 0$ then
\[
[\partial_I H_1\,\,and\,\,\partial_R (H_1-H_2)]\quad\mbox{or}\quad
[\partial_I H_1\,\, and\,\,\partial_I (H_1-H_2)]
\]
are $\mathbb R$-linearly independent.
\end{Lm}
\begin{proof}
As $M_1\ast M_2\not\equiv 0$ then in view of \eqref{eq:difference} $\partial_R(H_1-H_2)\not\equiv 0$ or
$\partial_I(H_1-H_2)\not\equiv 0$. Assume for example that $\partial_I(H_1-H_2)\not\equiv 0$. Assume that
\[
\alpha \partial_I H_1+\beta \partial_I(H_1-H_2) = 0
\]
where $(\alpha,\beta)\ne 0$, $\alpha,\beta\in\mathbb R$.
Restricting the equality above at $x=0$ and using that by \eqref{eq:H_1^R}-\eqref{eq:H_2^I} and
\eqref{eq:kappa}, $\partial_I(H_1-H_2)\big|_{x=0}=0$ and $\partial_I H_1\big|_{x=0}\ne 0$,
we obtain that $\alpha=0$. Hence, $\beta \partial_I(H_1-H_2)\equiv 0$. As $\partial_I(H_1-H_1)\not\equiv 0$ we
see that $\beta =0$. This shows that $\partial_I H_1$ and $\partial_I (H_1-H_2)$ are $\mathbb R$-linearly
independent. The case $\partial_R (H_1-H_2)\not\equiv 0$ is considered in the same way. 
\end{proof}

Theorem~\ref{Theorem5.2} and Theorem~\ref{Theorem5.3} are now used to prove Theorem~\ref{THM14}
stated in the introduction.

\begin{proof}[Proof of Theorem~\ref{THM14}]
As the case $N \in {\mathbb Z}_{\geq 1}$ is treated in the same way as $N = 0$ we concentrate on the
latter case only. Let $\zeta , \xi$ with $\zeta \not= \xi $ be arbitrary elements in ${\mathcal S}_p$.
It is to show that there exists a continuous path $\gamma^\ast : [0,1] \rightarrow {\mathcal S}_p$
with $\gamma ^\ast (0) = \zeta $ and $\gamma ^\ast (1) = \xi $. The path $\gamma ^\ast $ will be
constructed by deforming the straight line $\ell $, parametrized by
   \[ \gamma ^0 : [0,1] \rightarrow \iLR , t \mapsto (1 - t) \zeta + t \xi .
   \]
First let us observe that as the straight line $\ell $ is compact, Lemma~\ref{countinglemma}
implies that there exist a tubular neighborhood ${\mathcal U}_\ell $ of $\ell $,
   \[ {\mathcal U}_\ell := \{ \varphi \in \iLR\,|\,{\rm dist}(\varphi ,\ell )
      < \delta \}
   \]
for some $\delta > 0$ and an integer $R > 0$ so that for any $\varphi \in {\mathcal U}_\ell $,
the eigenvalues $\lambda ^+_n$ and $\lambda^-_n = \overline{\lambda ^+_n}$ of $L(\varphi )$
with $|n| > R$ are in the disk $D_n$ whereas the $4R + 2$ remaining eigenvalues $\lambda ^\pm_n, |n| \leq R$,
are contained in $B_R$. In addition, in view of Lemma \ref{dircountinglemma}, we can ensure that
for any $\varphi\in{\mathcal U}_\ell$ and for any $|n|>R$, $\mu_n\in D_n$, and the remaining
$2R+1$ Dirichlet eigenvalues  $\mu_n\in B_R$, $|n|<R$.
The path $\gamma ^0$ will be deformed within ${\mathcal U}_\ell $. Note that for any $|n| > R$,
either $\lambda ^+_n$ and $\lambda ^-_n$ are both simple periodic eigenvalues or $\lambda ^+_n$
is a real periodic eigenvalue with $m_g(\lambda ^+_n) = 2$ and $m_p(\lambda ^+_n) = 2$.
Hence to verify that a potential $\varphi\in {\mathcal U}_\ell $ is standard it suffices to study
the eigenvalues $\lambda ^\pm _n$ with $|n| \leq R$.

As by Theorem~\ref{Theorem1.4}, ${\mathcal S}_p$ is open and the endpoints $\zeta ,
\xi $ of $\gamma ^0$ are assumed to be
in ${\mathcal S}_p$ there exist open balls ${\mathcal V}_\zeta, {\mathcal V}_\xi $ in
${\mathcal S} _p \cap {\mathcal U}_\ell $ centered at $\zeta $ respectively $\xi $.

In a first step we apply Proposition~\ref{Proposition5.4}, based on Theorem~\ref{Theorem5.3},
to show that there exists a path $\gamma ^1 : [0,1] \rightarrow {\mathcal U}_\ell $ with
$\gamma ^1 (0) \in {\mathcal V}_\zeta $ and $\gamma ^1(1) \in {\mathcal V}_\xi $ so that
for any $\varphi $ on $\gamma ^1$, no periodic eigenvalue $\lambda ^\pm _n$ with $|n| \leq R$
has geometric multiplicity two. 
Note that the path $\gamma _\zeta : [0,1] \rightarrow {\mathcal V}_\zeta  \ [\gamma _\xi :
[0,1] \rightarrow {\mathcal V}_\xi ]$, connecting $\gamma _\zeta (0) = \zeta \ [\gamma _\xi(0) = \xi ]$ 
with $\gamma _\zeta (1) = \gamma ^1(0)\  [\gamma _\xi (1) = \gamma ^1(1)]$ by a straight line is
in ${\mathcal S}_p \cap {\mathcal U}_\ell $.

Then we apply Proposition~\ref{Proposition5.6}, based on Theorem~\ref{Theorem5.2}, to
show that $\gamma ^1$ can be deformed within ${\mathcal U}_\ell $ to a path $\gamma ^2$ with the
same end points as $\gamma ^1$ so that for any $\varphi $ on $\gamma ^2$, all its
periodic eigenvalues $\lambda ^\pm _n$ with $|n| \leq R$ are simple. In particular,
$\gamma ^2$ is contained in ${\mathcal S}_p$. The path $\gamma ^\ast $ is then defined
by concatenating $\gamma _\zeta, \gamma ^2$, and $\gamma ^{-1}_\xi$, i.e.,
$\gamma^\ast = \gamma ^{-1}_\xi \circ \gamma ^2 \circ \gamma _\zeta $. 

To describe our construction of $\gamma ^1$ in more detail we first introduce some more notation. Recall
that for any Dirichlet eigenvalue $\mu $ of $L(\varphi )$ with $\varphi \in L^2_c,
m_D(\mu ) \equiv m_D(\mu , \varphi )$ denotes its algebraic multiplicity. It is
convenient to set $m_D(\mu ) = m_D(\mu , \varphi ) = 0$ for any $\mu $ in ${\mathbb C}$
which is not a Dirichlet eigenvalue of $L(\varphi )$. Note that for any $\varphi \in
{\mathcal U}_\ell $ one has
   \[ m_D(\lambda ^\pm _n(\varphi )) \leq 2R + 1 \quad \forall |n| \leq R .
   \]
Furthermore introduce for any $\varphi \in {\mathcal U}_\ell $
   \[ M^D_\varphi := \max \{ m_D(\lambda ^\pm_n(\varphi ))\,|\,|n| \leq R ; \ m_g(\lambda
      ^\pm _n(\varphi )) = 2 \}.
   \]
We point out that
   \[ 0 \leq M^D_\varphi \leq 2R + 1 \quad \forall \varphi \in {\mathcal U}_\ell\,.
   \]
Finally, for any continuous path $\gamma : [0,1] \rightarrow {\mathcal U}_\ell$ set
   \[ M^D_\gamma := \mbox{max}\{ M^D_{\gamma (t)} \big\arrowvert 0 \leq t \leq 1 \}\,.
   \]
Note that $M^D_\gamma=0$ implies that for any $\varphi\in\gamma$ there is {\em no} periodic
eigenvalue $\lambda_n^\pm$ with $|n|\le R$ and $m_g(\lambda_n^\pm)=2$.
If $M^D_{\gamma ^0} = 0$, choose $\gamma ^1$ to be $\gamma ^0$. On the other hand, if
$M^D_{\gamma ^0} > 0$, then Proposition~\ref{Proposition5.4} says that there exists a
continuous path $\tilde \gamma ^0 : [0,1] \rightarrow {\mathcal U}_\ell $, connecting
${\mathcal V}_\zeta $ with ${\mathcal V}_\xi $ so that $M^D_{\tilde \gamma ^0} < M^D
_{\gamma ^0}$. In particular, $\tilde \gamma ^0(0) \in {\mathcal V}_\zeta $ and $\tilde
\gamma ^0(1) \in {\mathcal V}_\xi $. This procedure is iterated till we get a continuous
path $\gamma ^1 : [0,1] \rightarrow {\mathcal U}_\ell $ connecting ${\mathcal V}_\zeta $
with ${\mathcal V}_\xi $ so that $M^D_{\gamma ^1} = 0$.

To deform $\gamma ^1$ to $\gamma ^2$ we have to deal with potentials $\varphi $ with
multiple periodic eigenvalues $\lambda ^\pm _n(\varphi )$ of geometric multiplicity
one, i.e., $m_p(\lambda ^\pm _n(\varphi )) \geq 2$ and $m_g(\lambda ^\pm _n(\varphi)) = 1$.
To this end introduce for any $\varphi \in {\mathcal U}_\ell $
   \[ M^p_\varphi = \mbox{max}\{ m_p(\lambda ^\pm_n(\varphi ))\,|\,|n| \leq R \} .
   \]
As $1 \leq m_p(\lambda ^\pm _n(\varphi )) \leq 4R + 2$ for any $\varphi $ in
${\mathcal U}_\ell $ it follows that $1 \leq M^p_\varphi \leq 4R + 2$. Moreover,
by construction
   \[ M^p_{\gamma ^1(0)} = 1 \mbox { and } M^p_{\gamma ^1(1)} = 1 .
   \]
Finally, for a continuous path $\gamma : [0,1] \rightarrow {\mathcal U}_\ell $ define
   \[ M^p_\gamma := \mbox{max}\{ M^p_{\gamma (t)}\,|\,0 \leq t \leq 1 \} .
   \]
We now deform the path $\gamma ^1$. If $M^p_{\gamma ^1} = 1$, then $\gamma ^1 $ is
already a path in ${\mathcal S}_p$ and we set $\gamma ^2:= \gamma ^1$. On the other
hand, if $M^p_{\gamma ^1} \geq 2$, Proposition~\ref{Proposition5.6} implies that
there exists a continuous path $\tilde \gamma ^1 : [0,1] \rightarrow {\mathcal U}_\ell $ from
$\gamma ^1(0)$ to $\gamma ^1(1)$ so that $M^p_{\tilde \gamma ^1} <M^p_{\gamma ^1}$ and $M^D_{\tilde \gamma ^1} = 0$.
This procedure is iterated till we get a continuous path 
$\gamma^2 : [0,1] \rightarrow {\mathcal U}_\ell $ from $\gamma ^1(0)$ to $\gamma ^1(1)$ so
that $M^p_{\gamma ^2} = 1$. Then $\gamma ^2$ is a path inside ${\mathcal S}_p$
connecting $\gamma ^1(0)$ with $\gamma ^1(1)$.
\end{proof}

It remains to prove the two propositions used in the proof of Theorem \ref{THM14}.

\begin{pr}
\label{Proposition5.4} Let $\gamma : [0,1] \rightarrow {\mathcal U}_\ell $ be a
continuous path with standard potentials as end points, i.e.,
$\zeta := \gamma (0),$ $\xi := \gamma (1) \in {\mathcal S}_p$, and $\zeta \not=
\xi $. Denote by ${\mathcal V}_\zeta , {\mathcal V}_\xi $ open disjoint balls in
${\mathcal S}_p \cap {\mathcal U}_\ell $ centered at $ \zeta $, respectively
$\xi $. If $M^D_\gamma > 0$ then there exists a continuous path $\tilde \gamma :
[0,1] \rightarrow {\mathcal U}_\ell $ with $\tilde \gamma (0) \in {\mathcal V}
_\zeta , \tilde \gamma (1) \in {\mathcal V}_\xi $ and $M^D_{\tilde \gamma }
< M^D_\gamma $.
\end{pr}

\begin{proof} For any $\varphi \in \gamma $, denote by
   \[ \lambda ^1(\varphi ), \ldots , \lambda ^{K}(\varphi ) , \ {K}
      \equiv {K}_\varphi \in {\mathbb Z}_{\geq 0}
   \]
the list of different periodic eigenvalues of $L(\varphi )$ inside $B_R$ with
$m_g(\lambda ^k(\varphi )) = 2$ for any $1 \leq k \leq {K}$. If $M^D_\varphi < M^D_\gamma $,
then choose an open ball ${\mathcal W}_\varphi\subseteq{\mathcal U}_\ell$
centered at $\varphi$ so that for any $\psi \in {\mathcal W}_\varphi $,
\begin{equation}\label{5.30'}
M^D_\psi \leq M^D_\varphi \ (< M^D_\gamma)\,.
\end{equation}
The existence of such neighborhood follows easily from the second statement of 
Theorem \ref{Theorem5.2} and Theorem \ref{TheoremA.1} (i) applied with $\chi=\chi_D$.
On the other hand, if $M^D_\varphi = M^D_\gamma $, let
   \[ I \equiv I_\varphi := \{ 1 \leq j \leq {K}\,|\,m_D(\lambda
      ^j(\varphi )) = M^D_\gamma \} .
   \]
By Theorem~\ref{Theorem5.3}, applied to $( \lambda ^j{(\varphi)}, \varphi)$ for
any $j \in I,$ there exists a path connected neighborhood ${\mathcal W}_\varphi $ of $\varphi $ in
${\mathcal U}_\ell $ and a union ${\mathcal Z}_\varphi = \cup _{j \in I_\varphi }
{\mathcal Z}^j_\varphi $ of submanifolds ${\mathcal Z}^j_\varphi $ of codimension
two which are closed in ${\mathcal W}_\varphi $ so that
   \begin{equation}
   \label{5.40} \{ \psi \in {\mathcal W}_\varphi\,|\,M^D_\psi = M^D_\gamma
                \} \subseteq {\mathcal Z}_\varphi .
   \end{equation}
By shrinking ${\mathcal W}_\varphi $, if necessary, we further can assume that
   \begin{equation}
   \label{5.41} M^D_\psi \leq M^D_\varphi \ (= M^D_\gamma ) \quad \forall \psi \in
                {\mathcal W}_\varphi .
   \end{equation}
In addition, if $\varphi $ is either $\zeta $ or $\xi $ we assume that 
${\mathcal W}_\varphi \subseteq {\mathcal V}_\zeta$ or
${\mathcal W}_\varphi \subseteq {\mathcal V}_\xi$ respectively.
As $\{ \gamma (t) \big\arrowvert 0 \leq t \leq 1\} $ is compact the cover
$({\mathcal W}_\varphi )_{\varphi \in \gamma ([0,1])}$
admits a finite subcover $({\mathcal W}_\varphi )_{\varphi \in \Lambda }$ where
$\Lambda \subseteq \gamma ([0,1])$ is finite and contains $\zeta $ and $\xi $. We claim
that there exists a sequence ${\mathcal W}_i \equiv {\mathcal W}_{\varphi _i}, 1
\leq i \leq N \equiv N_\gamma $ with $\varphi _i \in \Lambda $ so that ${\mathcal W}_{i
-1} \cap {\mathcal W}_i \not= \emptyset $ for any $2 \leq j \leq N_\gamma $ and $\varphi_1=\zeta$ or
$\varphi_N=\xi$. Indeed, choose $\varphi _1= \zeta $ to begin with. 
Then $\gamma ^{-1}({\mathcal W}_1)$ is an open subset of $[0,1]$.
As ${\mathcal V}_\zeta \cap {\mathcal V}_\xi = \emptyset $ it follows that
   \[ t_1:= \sup\gamma^{-1}({\mathcal W}_1) < 1 .
   \]
As $({\mathcal W}_\varphi )_{\varphi \in \Lambda }$ covers $\gamma ([0,1])$ there
exists $\varphi _2 \in \Lambda $ with $\gamma (t_1) \in {\mathcal W}_2$. As ${\mathcal W}
_2$ is open and $\gamma : [0,1] \rightarrow {\mathcal U}_\ell $ is continuous it then
follows that ${\mathcal W}_1 \cap {\mathcal W}_2 \not= \emptyset $. Continuing in this
way one obtains the sequence ${\mathcal W}_i$ with $\varphi _i \in \Lambda , 1 \leq i \leq
N \equiv N_\gamma $ so that $\varphi _1 = \zeta , \ \varphi _N = \xi $, and ${\mathcal W}
_{i-1} \cap {\mathcal W}_i \not= \emptyset $ for any $2 \leq i \leq N$. As for any $1
\leq i \leq N, {\mathcal Z}_i \equiv {\mathcal Z}_{\varphi _i}$ is a finite union of
submanifolds of codimension two it then follows that for any $2 \leq i \leq N, \
({\mathcal W}_{i - 1} \cap {\mathcal W}_i) \backslash ({\mathcal Z}_{i-1} \cup {\mathcal Z}_i)
\not= 0$.

For any $2 \leq i \leq N$, choose $\eta _i \in ({\mathcal W}_{i-1} \cap {\mathcal W}_i)
\backslash ({\mathcal Z}_{i-1} \cup {\mathcal Z}_i)$. By \eqref{5.30'}-\eqref{5.41}
one concludes that for any $2 \leq i \leq N, \ M^D_{\eta _i} < M^D_\gamma $ and
   \[ \eta _i, \eta _{i+1} \in {\mathcal W}_i \backslash {\mathcal Z}_i .
   \]
As ${\mathcal Z}_i$ is a finite union of submanifolds of codimension two which are
closed in ${\mathcal W}_i$, Lemma~\ref{Lemma5.5} stated below applies repeatedly.
Hence for any $2 \leq i \leq N - 1$,
there exists a continuous path $\gamma _i : [0,1] \rightarrow {\mathcal W}_i
\backslash {\mathcal Z}_i$ such that $\gamma _i(0) = \eta_i $ and $\gamma _i
(1) = \eta _{i+1}$. By \eqref{5.30'}-\eqref{5.41} one has $M^D_{\gamma _i} < M^D_\gamma $.
As $\eta _2 \in {\mathcal W}_1 \subseteq {\mathcal V}_\zeta $ and $\eta _N \in
{\mathcal W}_N \subseteq {\mathcal V}_\xi $ it then follows that the concatenation
$\tilde \gamma $ of $\gamma _ 2, \ldots ,\gamma _{N-1}$ is a continuous curve $\tilde
\gamma : [0,1] \rightarrow {\mathcal U}_\ell $ with the properties listed in
Proposition~\ref{Proposition5.4}.
\end{proof}

Let us now state and prove the lemma referred to in the proof of
Proposition~\ref{Proposition5.4}.

\begin{Lm}
\label{Lemma5.5} Let ${\mathcal U}$ be an open, path connected set in a Hilbert space
$E$ and let ${\mathcal Z} \subseteq {\mathcal U}$ be a closed smooth submanifold of
codimension two. Then ${\mathcal U} \backslash {\mathcal Z}$ is open and path
connected.
\end{Lm}

\begin{proof} This lemma is well known. In fact, it can be proved following
the line of arguments used in Proposition~\ref{Proposition5.4} and by taking into account the following
special case, where ${\mathcal Z}$ is a linear subspace of $E$ of codimension two and
hence $E \backslash {\mathcal Z}$ is obviously path connected.
\end{proof}

\begin{pr}\label{Proposition5.6}
Let $\gamma : [0,1] \rightarrow {\mathcal U}_\ell $ be a
continuous path with standard potentials as end points so that $M^D_\gamma = 0$. If
$M^p_\gamma \geq 2$, then there exists a continuous path $\tilde \gamma : [0,1]
\rightarrow {\mathcal U}_\ell $ with the same end points as $\gamma$, 
$M^D_{\tilde \gamma } = 0$, and $M^p_{\tilde \gamma } < M^p_\gamma $.
\end{pr}

\begin{proof} The assumption $M^D_\gamma = 0$ implies that $M^D_{\gamma (0)} = 0$ and
$M^D_{\gamma (1)} = 0$. As $\gamma (0)$ and $\gamma (1)$ are both standard potentials
one concludes from Proposition \ref{pr2} that all their periodic eigenvalues $\lambda ^\pm _n$ with 
$|n| \leq R$ are simple. In particular, one has $M^p_{\gamma (0)} = 1$ and $M^p_{\gamma (1)}
= 1$. For any $\varphi \in \gamma $ denote by
   \[ \lambda ^1(\varphi ), \ldots , \lambda ^K(\varphi ) , \ K
      \equiv K _\varphi \in {\mathbb Z}_{\geq 0}
   \]
the list of different multiple periodic eigenvalues of $L(\varphi )$ inside $B_R$.
By assumption, $m_g(\lambda ^i(\varphi )) = 1$ for $1 \leq i \leq K$. If
$M^p_\varphi < M^p_\gamma$, by the second statement of Theorem~\ref{Theorem5.2} and
Theorem \ref{TheoremA.1} (i) applied with $\chi=\chi_p$ and $\chi=\chi_D$ there exists a neighborhood
${\mathcal W}_\varphi \subseteq {\mathcal U}_\ell $ of $\varphi $ so that for any
$\psi \in {\mathcal W}_\varphi$
   \[ M^p_\psi \leq M^p_\varphi (< M^p_\gamma ) \mbox { and } M^D_\psi = 0 .
   \]
On the other hand, if $M^p_\varphi = M^p_\gamma $, let
   \[ I \equiv I_\varphi := \{ 1 \leq j \leq K \big\arrowvert m_p(\lambda
      ^j (\varphi )) = M^p_\gamma \} .
   \]
By Theorem~\ref{Theorem5.2}, applied to $(\varphi , \lambda ^j(\varphi ))$ for
any $j \in I$, there exists an open ball ${\mathcal W}_\varphi$ in ${\mathcal U}_\ell$,
centered at $\varphi$ and a union ${\mathcal Z}_\varphi = \cup _{j \in I_\varphi }
{\mathcal Z}^j_\varphi $ of submanifolds ${\mathcal Z}^j_{\varphi }$ of
codimension two which are closed in ${\mathcal W}_\varphi $ so that
   \[ \{ \psi \in W_\varphi\,|\,M^p_\psi = M^p_\gamma \} \subseteq {\mathcal Z}
      _\varphi .
   \]
By shrinking $W_\varphi $, if necessary, we further can assume that for any $\psi
\in {\mathcal W}_\varphi $
   \[ M^p_\psi \leq M^p_\varphi \mbox { and } M^D_\psi = 0 .
   \]
Then argue as in the proof of Proposition~\ref{Proposition5.4} to conclude that
there is a continuous path $\tilde \gamma : [0,1] \rightarrow {\mathcal U}_\ell $
with the same end points as $\gamma$, $M^D_{\tilde \gamma } = 0$, and
$M^p_{\tilde \gamma } < M^p_\gamma $.
\end{proof}

It turns out that the proof of Theorem \ref{THM14} actually leads to the following
additional result. We say that a potential $\varphi \in L^2_c$ is {\em $R$-simple}, 
$R \in{\mathbb Z}_{\geq - 1}$, if $\mu _n,\lambda ^\pm _n \in D_n$ for any $|n| >R$,
$\mu _n,\lambda ^\pm _n \in B_R$ for any $|n| \leq R$, and the eigenvalues
$(\lambda ^+_n,\lambda ^-_n)_{|n|\leq R}$ are all simple. 
Note that the zero potential is $(-1)$-simple. 
Denote by ${\mathcal T}^R$ the set of $R$-simple potentials in $L^2_c$ and by
${\mathcal T}$ the set of potentials $\varphi\in L^2_c$ so that $\spec_p L(\varphi)$ is simple.

Inspecting the proof of Theorem \ref{THM14} one sees that at the same time, the following result has been
proved.
\begin{cor}\label{Corollary3.7} 
For any $N \in {\mathbb Z}_{\ge 0}$ and for any  $\zeta, \xi\in{\mathcal T}\cap i H^N_r$ there exists
$R\in{\mathbb Z}_{\ge-1}$ such that $\zeta$ and $\xi$ are path connected in ${\mathcal T}^R\cap i H^N_r$.
\end{cor}
\begin{proof} 
First, consider the case $N=0$.
Take $\zeta, \xi \in{\mathcal T}\cap i L^2_r$ and denote by $\ell$ the straight line connecting $\zeta$ with
$\xi$ in $i L^2_r$.
By the counting lemmas and the fact that $\ell$ is compact, there exists $R\in{\mathbb Z}_{\ge-1}$ so that
for any $\varphi\in\ell$, $\mu_n,\lambda_n^\pm\in D_n$ for any $|n|>R$, and $\mu_n,\lambda_n^\pm\in B_R$,
for any $|n|\le R$. With this choice of $R$, one can follow the arguments of the proof of Theorem \ref{THM14} to
conclude that there exists a path $\gamma$ in ${\mathcal T}^R\cap i L^2_r$ connecting $\zeta$ and $\xi$.
The case $N\in{\mathbb Z}_{\ge 1}$ is treated similarly.
\end{proof}

For later applications it is useful to consider also a stronger version of the notion of $R$-simple potentials. 
For any $\psi \in i L^2_r$ denote by $\Iso_0(\psi)$ the connected component containing $\psi$ of the 
{\em isospectral set}, 
\[
\Iso(\psi):=\big\{\varphi\in i L^2_r\,|\,\spec_p L(\varphi)=\spec_p L(\psi)
\big\}\,.\footnote{The periodic eigenvalues are counted with their algebraic multiplicities.}
\]
We say that a potential $\psi \in  i L^2_r$ is {\em uniformly $R$-simple},
$R \in {\mathbb Z}_{\geq - 1}$, if $\Iso_0(\psi)\subseteq{\mathcal T}^R$.
In other words, we require that  for any $\varphi \in \Iso_0(\psi)$, the periodic and the Dirichlet eigenvalues
of $L(\varphi)$ are contained in $B_R \cup \bigcup _{|n| > R} D_n$ and, when counted with their algebraic
multiplicities, satisfy the following conditions:
\begin{description}
\item[{\rm (S1)}] $\#(D_n \cap \spec_p L(\varphi )) = 2$, $\#(D_n\cap\spec_D L(\varphi )))=1\,\,\forall |n| > R$;
\item[{\rm (S2)}] $\#(B_R \cap \rm{spec}_p L(\varphi ))=4R+2$,
$\#(B_R \cap \rm{spec}_D L(\varphi ))=2R+1$;
\item[{\rm (S3)}] The ball $B_R$ contains only {\em simple} periodic eigenvalues.
\end{description}
We denote the set of uniformly $R$-simple potentials by ${\mathcal U}^R$ and let 
${\mathcal U}^*:=\bigcup_{R \ge -1} {\mathcal U}^R$. 
In view of the Counting Lemmas (see Lemma \ref{countinglemma} and Lemma \ref{dircountinglemma}) and
the compactness of $\Iso_0(\psi)$  (see Lemma \ref{lem:iso-compact} below) it follows that
for any $N\in{\mathbb Z}_{\geq 0}$, 
\begin{equation}\label{eq:simple_uniform}
{\mathcal T}\cap i H^N_r\subseteq{\mathcal U}^*\cap i H^N_r.
\end{equation}
\begin{pr}\label{prop:simple_dense} 
For any $N\in{\mathbb Z}_{\geq 0}$, the set ${\mathcal T}\cap i H^N_r$ is dense in $i H^N_r$. 
As a consequence, ${\mathcal U}^*\cap i H^N_r$ is dense in $i H^N_r$.
\end{pr}
\begin{proof} 
Let $N\in{\mathbb Z}_{\ge 0}$ and let $\psi\in i H^N_r$. In view of the Counting Lemmas there exists 
$R\in{\mathbb Z}_{\ge -1}$ and an open neighborhood $U(\psi)$ of $\psi$ in $i H^N_r$ such that
any $\varphi\in U(\psi)$ satisfies conditions {\rm (S1)} and {\rm (S2)}. It follows from Theorem \ref{Theorem5.2}
and Theorem \ref{Theorem5.3} that ${\mathcal T}^R\cap U(\psi)$ is open and dense in $U(\psi)$.
In view of Proposition \ref{pr2}, Lemma \ref{perdir}, and Theorem \ref{Theorem5.3}, for any $|n|>R$, the set
\[
Z_n:=\big\{\varphi\in U(\psi)\,\big|\,\lambda^-_n(\varphi)=\lambda^+_n(\varphi)\big\}\subseteq U(\psi)
\]
is contained in a submanifold in $U(\psi)$ of (real) codimension two.
Hence, for any $|n|>R$, $Z_n$ is closed and nowhere dense in ${\mathcal T}^R\cap U(\psi)$, and
by the Baire theorem the set 
\[
{\mathcal T}\cap U(\psi)=\Big({\mathcal T}^R\cap U(\psi)\Big)\setminus\bigcup\limits_{|n|>R}Z_n
\]
is dense in $U(\psi)$. This completes the proof of the first statement of Proposition \ref{prop:simple_dense}.
The second statement then follows from \eqref{eq:simple_uniform}.
\end{proof}
We finish this appendix by showing that ${\mathcal U}^*\cap i H^N_r, N \geq 1,$ contains
a subset which is open in $ iH^N_r$. Here we use that for $N\ge 1,$ the Counting Lemmas
for the periodic and the Dirichlet spectrum of $L(\varphi )$ with $\varphi $ in $iH^N_r$
hold uniformly on any bounded subset of $iH^N_r$.
\begin{pr}\label{Proposition 3.9} 
$\forall N \in {\mathbb Z}_{\geq 1}$, ${\mathcal U}^*\cap i H^N_r$
contains a subset which is open and dense in $iH^N_r$.
\end{pr}
We first need to make some preliminary considerations.
It follows from \cite[Theorem 13.4 and 13.5]{GKP2} that  the quantities 
$ {\mathcal J}_1(\varphi) := \int ^1_0 |\varphi _1 |^2\,dx$ and 
${\mathcal J}_2(\varphi) := \int ^1_0 (|\partial _x \varphi _1|^2 - |\varphi _1|^4)\,dx$ 
are spectral invariants of the periodic spectrum of $L(\varphi )$ for $\varphi\in i H^1_r$. 
By the generalized Gagliardo-Nirenberg inequality there exist absolute constants
$C_1, C_2 > 0$ so that for any $u \in L^2({\mathbb T}, {\mathbb C})$ and $\varepsilon>0$
\[ 
\| u\| _{L^4} \le C_1 \|\partial _x u \| ^{1/4} \| u\|^{3/4} + C_2\|u\|
\le C_1\big(\varepsilon^2\|\partial_x u\|^{1/2}+\|u\|^{3/2}/\varepsilon^2\big)+ C_2\|u\|
\]
-- see \cite[Theorem 1]{Nir} with $n = 1$ (for the case of a circle instead of an interval), 
$j = 0, p = 4, m = 1, r = 2, q = 2,$ and $a = 1/4$.
By taking $\varepsilon^2=1/(3C_1)$ in the inequality above one obtains
\[ 
\|u\|^4_{L^4} \le \frac{1}{3}\,\|\partial_x u\|^2 + 3^7C_1^8\|u\|^6 + 3^3 C_2^4\| u\|^4 .
\]
Writing $\int ^1_0 |\partial _x \varphi _1|^2\,dx={\mathcal J}_2 + \int ^1_0 |\varphi _1 |^4\,dx$
it then follows that
$\frac{2}{3}\,\int ^1_0 |\partial _x \varphi _1|^2 dx \le {\mathcal J}_2 + 
C({\mathcal J}^3_1 + {\mathcal J}^2_1)$ where  $C=\max (3^7 C_1^8, 3^3 C_2^4)$. 
We thus have proved that any $\varphi \in \Iso_0(\psi)\cap i H^1_r$, can be bounded by
\begin{equation*}
\|\varphi _1\| ^2_{H^1}\le  
\big(3 {\mathcal J}_2 + 3C({\mathcal J}^3_1 + {\mathcal J}^2_1) + {\mathcal J}_1\big)\big\arrowvert _\psi 
\le 3 \| \psi _1 \|^2_{H^1} + 3C \big( \| \psi _1\| ^6 + \| \psi _1 \| ^4 \big) .
\end{equation*}
where $\|\psi_1\|^2_{H^1}:=\|\partial_x\psi_1\|^2+\|\psi_1\|^2$.
This implies that for any $\rho>0$ there exists a positive constant $C_\rho>0$ such that for any 
$\psi\in{\mathcal B}^1_\rho$,
\begin{equation}\label{eq:iso_inclusion}
\Iso_0(\psi)\cap i H^1_r\subseteq{\mathcal B}^1_{C_\rho}\,,
\end{equation}
where ${\mathcal B}^N_\rho:=\{\varphi\in i H^N_r\,|\,\|\varphi\|_{H^N_c}<\rho\}$.

\noindent{\em Proof of Proposition \ref{Proposition 3.9}. }
Take $N\in{\mathbb Z}_{\ge 1}$, $\rho>0$, and consider the set
\[
{\mathcal I}^N_\rho:=\bigcup_{\varphi\in{\mathcal B}^N_\rho}\Iso_0(\varphi)\cap i H^N_r\,.
\]
In view of \eqref{eq:iso_inclusion} we have the following sequence of continuous embedding
\[
{\mathcal I}^N_\rho\subseteq{\mathcal I}^1_\rho\subseteq{\mathcal B}^1_{C_\rho}\subseteq i L^2_r
\]
where the last embedding is compact. Hence, ${\mathcal I}^N_\rho$ is a precompact set in $i L^2_r$ and, 
by the Counting Lemma, there exists $R\equiv R_\rho\ge 0$ such that any $\varphi\in{\mathcal I}^N_\rho$
satisfies conditions {\rm (S1)} and  {\rm (S2)}. This implies that 
${\mathcal T}^R\cap{\mathcal B}^N_\rho\subseteq{\mathcal U}^*\cap i H^N_r$. Hence,
\begin{equation}\label{eq:inclusion}
{\mathcal T}\cap{\mathcal B}^N_\rho\subseteq{\mathcal P}_\rho\subseteq{\mathcal U}^*\cap i H^N_r,
\end{equation}
where ${\mathcal P}_\rho$ denotes the open subset ${\mathcal T}^R\cap{\mathcal B}^N_\rho$ of $i H^N_r$.
By Proposition \ref{prop:simple_dense}, the set ${\mathcal T}\cap{\mathcal B}^N_\rho$ is dense in 
${\mathcal B}^N_\rho$. This together with the first inclusion in \eqref{eq:inclusion} implies that
${\mathcal P}_\rho$ is open and dense in ${\mathcal B}^N_\rho$. 
Therefore $\bigcup_{\rho\in{\mathbb Z}_{\ge 1}} {\mathcal P}_\rho $ is open and dense in $iH^N_r$ and
thus is a subset of ${\mathcal  U}^* \cap i H^N_r$ with the claimed properties. 
\hfill$\square$\\

\noindent We complete this section by proving the following result used above to establish \eqref{eq:simple_uniform}.
\begin{Lm}\label{lem:iso-compact} 
For any $\psi $ in $i L^2_r$, $\rm{Iso}(\psi )$ is compact.
\end{Lm}
\begin{proof} 
Let $(\varphi_n )_{n \geq 1}$ be a sequence in $\Iso(\psi )$. By \cite[Theorem 13.4]{GKP2} the $L^2$-norm is
a spectral invariant of the periodic spectrum of $L(\varphi)$ for $\varphi\in i L^2_r$. Hence, $\|\varphi _n\|=\|\psi\|$, 
for any $n\ge 1$, and therefore there exist $\varphi\in i L^2_r$ and a weakly convergent subsequence, which we again
denote by $(\varphi_n)_{n\ge 1}$, such that $\varphi_n\stackrel{w}{\to}\varphi$ as $n\to\infty$.
By \cite[Theorem 4.1]{GKP2}, for any $\lambda\in{\mathbb C}$ the map $L^2_c\to{\mathbb C}$,
$\varphi\mapsto\Delta(\lambda,\varphi)$, is compact and hence 
$\Delta(\lambda,\psi)=\lim_{n\to\infty}\Delta(\lambda,\varphi_n)=\Delta (\lambda,\varphi )$.
By the definition of $\Iso(\psi )$ it then follows that $\varphi \in \Iso(\psi )$, implying that $\|\varphi\|=\|\psi\| $. 
As a consequence $\varphi_n\to\varphi$ in $ iL^2_r$. This shows that $\Iso(\psi )$ is a compact subset of $i L^2_r$.
\end{proof}

%%%%%%%%%%%%%%%%%%%%%%%%%%%%%%%%%%%%%%%%%%%%%%%%%%%%%%%%%%%%%%%%%%%%%%%%%%

%%%%%%%%%%%%%%%%%%%%%%%%%%%%%%%%%%%%%%%%%%%%%%%%%%%%%%%%%%%%%%%%%%%%%%%

\section{Proof of Theorem \ref{Theorem1.4}}\label{Proof of Theorem1.4}

The aim of this section is to prove Theorem~\ref{Theorem1.4}.

\begin{proof}[Proof of Theorem~\ref{Theorem1.4}] We begin by showing that
${\mathcal S}_p \cap iH^N_r$ and ${\mathcal S}_D \cap iH^N_r$ are open. First
consider the case ${\mathcal S}_p \cap iH^0_r = {\mathcal S}_p$. For $\psi \in {\mathcal S}_p$
arbitrary choose $R \in {\mathbb Z}_{\geq 0}$ as in Lemma~\ref{countinglemma} and let $\varepsilon > 0$ be
smaller than twice the distance between any two different periodic eigenvalues of $L(\psi)$ in $B_R$.
By Lemma~\ref{countinglemma}, Proposition~\ref{rootspr2.31}, and Theorem \ref{TheoremA.1}(i),
there exists an open neighborhood $W_\psi $ of $\psi $ in $L^2_c$ so that for any $\varphi \in W_\psi $,
$\#\big(D_n \cap \spec_p L(\varphi)\big) = 2$ and $\#\big(B_R\cap\spec_p L(\varphi)\big) = 4R + 2$ and
for any eigenvalue $\lambda \in B_R \cap\spec_p L(\psi)$,
\[
\#\big(D^\varepsilon(\lambda)\cap\spec_p L(\varphi)\big) = 
              \#\big(D^\varepsilon(\lambda )\cap\spec_p L(\psi)\big)
\]
where $D^\varepsilon (\lambda ) \subseteq {\mathbb C}$ is the open disk of radius $\varepsilon$ centered
at $\lambda $. By the definition of ${\mathcal S}_p$, one then concludes that
$W_\psi \cap \iLR \subseteq {\mathcal S}_p$. The openness of ${\mathcal S}_p
\cap i H^N_r(N \geq 1)$ and ${\mathcal S}_D \cap i H^N_r(N \geq 0)$ is proved
in a similar fashion. Next we show that ${\mathcal S}_p$ is dense in $\iLR$.
As $iH^1_r$ is dense in $\iLR$ it suffices to show that ${\mathcal S}_p
\cap iH^1_r$ is dense in $iH^1_r$. Actually we prove a slightly stronger
statement.
Denote by
$\mathcal{B}^1_\rho$ the open ball of radius $\rho > 0$ in $H_c^1$, centered at $0$.
Choose $R \equiv R_\rho \geq 1$ as in Corollary \ref{corcount} and introduce
\[
\mathcal{S}_{p,R}:=\left\{\varphi \in \mathcal{S}_p\,|\,\lambda_n^\pm \text{ simple }
\forall \, |n| \leq R\right\}.
\]
We claim that $\mathcal{S}_{p,R} \cap
\mathcal{B}_\rho^1$ is dense in $\mathcal{B}^1_\rho \cap \iLR$.
For any $\varphi \in \mathcal{B}_\rho^1$ and $
R$ as above introduce
\begin{equation*}
\mathcal{Q}_{p,R}(\lambda):= \prod _{|k|\leq R }(\lambda_k^+-\lambda)(\lambda_k^--\lambda).
\end{equation*}
Then $\mathcal{Q}_{p,R}$ is a polynomial in $\lambda$ of degree $4R+2$ with coefficients
depending analytically in $\varphi$ on  $\mathcal{B}_\rho^1$.
Indeed, any coefficient of the polynomial $\mathcal{Q}_{p,R}$ is a symmetric polynomial
in $\lambda_k^\pm$, $|k| \leq R$, and hence can be written as a polynomial  in
\[
s_n:=\sum_{|k|\leq R} (\lambda_k^+)^n +(\lambda_k^-)^n, \quad  0 \leq n \leq 4R+2.
\]
To see that each $s_n$ is analytic on $\mathcal{B}_\rho^1$, note that by the argument principle,
$s_n$ is given by
\begin{equation*}
s_n= \frac{1}{2 \pi i} \int_{|\lambda|= \pi (R+1/4)} \lambda^n
\frac{\dot \chi_p(\lambda)}{\chi_p(\lambda)}\, d\lambda
\end{equation*}
and hence analytic as $\chi_p(\lambda, \varphi)$ and $\dot \chi_p(\lambda, \varphi)$
are analytic on $\mathbb C \times \Lc$ and $\chi _p (\lambda , \varphi )$ does not vanish
for $\varphi \in {\mathcal B}^1_\rho $ and $\lambda$ in 
$\{\lambda\in{\mathbb C}\,|\,|\lambda|=R\pi + \pi /4 \}$. 
Denote by ${\mathcal D}_{p,R}$ the {\em discriminant} of the polynomial $\mathcal{Q}_{p,R}$. 
Recall that the discriminant is the resultant of $\mathcal{Q}_{p,R}$ and its derivative
$\partial_\lambda \mathcal{Q}_{p,R}(\lambda)$.
It is given by
\begin{equation*}
\mathcal{D}_{p,R}= \det \left( \begin{array}{cccccccc} a_0 & \cdots & a_{4R} &a_{4R+1}
& a_{4R+2}& 0& \cdots & 0 \\
\vdots &     & \vdots & \vdots &  \vdots &  \vdots &  &\vdots \\
0   & \cdots & a_0 &   a_1     &  a_2    & a_3& \cdots &a_{4R+2}  \\
b_0 & \cdots & b_{4R}& b_{4R+1}&  0      &  0 & \cdots & 0 \\
\vdots &     & \vdots & \vdots & \vdots & \vdots &     & \vdots \\
0 & \cdots   & 0 & b_0& b_1 & b_2 & \cdots  & b_{4R+1}    \end{array} \right )
\end{equation*}
where $(a_0,a_1,...,a_{4R+2})$ is the coefficient vector of $\mathcal{Q}_{p,R}$ and repeated
$4R+1$ times whereas $(b_0,b_1,...,b_{4R+1})$ is the one of $\dot{\mathcal{Q}}_{p,R}$,
and repeated $4R+2$ times. Note that $\mathcal{D}_{p,R}$
is an analytic function on $\mathcal{B}_\rho ^1$. Furthermore, it has the property that
it vanishes at an element $\varphi \in \mathcal{B}_\rho^1$ iff $\mathcal{Q}_{p,R}(\cdot,
\varphi)$ has at least one multiple zero -- see e.g. \cite{Ke}.
In particular, we have
\[
\left\{\varphi \in \mathcal{B}_\rho^1 \cap i H^1_r \,|\, \mathcal{D}_{p,R}(\varphi)\not=0
\right\}= \mathcal{S}_{p,R} \cap \mathcal{B}_\rho ^1.
\]
To show that $\mathcal{S}_{p,R} \cap \mathcal{B}_\rho ^1$ is dense in $\mathcal{B}_\rho^1
\cap i H^1_r$ it thus suffices to show that $\mathcal{S}_{p,R} \cap \mathcal{B}_\rho ^1 \ne \emptyset$ and
that $\mathcal{D}_{p,R}$ is real valued on $\mathcal{B}_\rho ^1 \cap iH^1_r$.
To see that $\mathcal{D}_{p,R}$ is real valued on $\mathcal{B}_\rho^1 \cap i H^1_r$ note that
by Proposition \ref{pr2} for any $k \in \mathbb Z$ and $\varphi $ in $\mathcal{B}_\rho^1\cap iH^1_r$, 
$\lambda_k^-(\varphi)= \bar \lambda_k^+(\varphi)$.  Hence by the definition of $\mathcal{Q}_{p,R}$,
its coefficients are real valued and $\mathcal{D}_{p,R}$ is therefore
real valued on $\mathcal{B}_\rho^1 \cap iH^1_r$. Finally, to see that $\mathcal{S}_{p,R}
\cap \mathcal{B}_\rho ^1\not= \emptyset$ we use that elements in $iH^1_r$ near $0$ can be
represented by Birkhoff coordinates. Indeed, by Theorem 1.1 of \cite{KLTZ} and formula
(3.8) of \cite{KLTZ}, for any element $\varphi$ of $iH^1_r$ near $0$ with Birkhoff
coordinates  $(x_k, y_k)_{k \in \mathbb Z}$ satisfying
for any given $k\in \mathbb Z$, $x_k^2+y_k^2\not= 0$, one has $\lambda_k^+ \not =\lambda_k^-$. 
Thus by Theorem 1.1 of \cite{KLTZ}, any sequence $(x_k,y_k)_{k \in  \mathbb Z}$
with values in $i{\mathbb R}\times i{\mathbb R}$ and $\sum_{k \in \mathbb Z}(1+|k|)^2(|x_k|^2+|y_k|^2)$ sufficiently
small so that $x_k^2+y_k^2\not=0$ for any $|k|\leq R$ is an element in 
$\mathcal{S}_{p,R} \cap\mathcal{B}_\rho ^1$. In the same way one shows that
${\mathcal S}_p \cap iH^N_r$ is dense in $i H^N_r$ for any $N \in {\mathbb Z}_{\ge 0}$.

The corresponding density result for ${\mathcal S}_D$ is proved in a similar fashion.
As $iH^1_r$ is dense in $iL^2_r$, it suffices to show that $\mathcal{S}_D\cap iH_r^1$ is dense in $i H^1_r$.
For any $\rho>0$ choose $R\equiv R_\rho\ge 1$ so that for any $\varphi\in\mathcal{B}_\rho^1$ the statement
of Lemma \ref{dircountinglemma} holds. Introduce
\[
\mathcal{S}_{D,R}:= \left\{ \varphi \in \iLR\,|\,\mu_n \text{ simple and } \mu_n\ne{\bar\mu}_n \, \, 
\forall |n| \, \leq R \right\}.
\]
We claim that for any $\rho> 0$, $\mathcal{S}_{D,R} \cap \mathcal{B}_\rho^1$ is dense in 
$\mathcal{B}_\rho^1\cap i H^1_r$.
Arguing as above one reduces in a first step the proof of the
density of $\mathcal{S}_{D,R} \cap \mathcal{B}_\rho^1$ in $\mathcal{B}_\rho^1 \cap i H^1_r$ to the proof of
$\mathcal{S}_{D,R} \cap \mathcal{B}_\rho ^1 \not = \emptyset$.
Here
\[
 \mathcal{Q}_{D,R}:= \prod_{|k|\leq R} (\mu_k- \lambda)({\bar\mu}_k -\lambda) 
\]
plays the role of $\mathcal{Q}_{p,R}$. Next we use again that by Theorem 1.1 in \cite{KLTZ}, any sequence
$(x_k, y_k)_{k \in \mathbb Z}$ with values in $i\mathbb R \times i\mathbb R$ and 
$\sum_{k\in \mathbb Z}(1+|k|)^2(|x_k|^2+|y_k|^2)$ sufficiently small, represents an element $\varphi$
in $\mathcal{B}_\rho ^1$ close to $0$. Together with Proposition 4.1 in \cite{KLTZ} it
follows that if $x_k \not = 0$ and $y_k = 0$, then $\lambda_k^- \not = \lambda_k^+$ and
$\mu_k  \in \left\{\lambda_k^+, \lambda_k^- \right\}$. Hence $ \mu_k \not \in  \mathbb R$ and $\mu_k \ne{\bar\mu}_k$.
In addition,  for $\sum_{k \in \mathbb Z}(1+|k|)^2(|x_k|^2+|y_k|^2)$ sufficiently small, 
$\mu_k,\lambda_k^\pm\in D_k$ for any $k \in {\mathbb Z}$.
Hence, any sequence
$(x_k, y_k)_{k \in \mathbb Z}$ with values in $i\mathbb R \times i\mathbb R$ and 
$\sum_{k\in \mathbb Z}(1+|k|)^2(|x_k|^2+|y_k|^2)$ sufficiently small, so that $y_k = 0$ for any $|k |
\leq R$, represents an element in $\mathcal{S}_{D,R} \cap i H_r^1 $. The case
$N \in {\mathbb Z}_{\geq 1}$ is treated in a similar way.
\end{proof}

\medskip

Inspecting the proof of Theorem~\ref{Theorem1.4} one sees that actually the following
result for the set of potentials ${\mathcal S}^\ast $ introduced at the end of
Section 3 has been proved.

\begin{cor}
\label{Corollary4.1} For any $N \in {\mathbb Z}_{\geq 0}$, ${\mathcal S}^\ast 
\cap iH^N_r$ is dense in $iH^N_r$.
\end{cor}

%%%%%%%%%%%%%%%%%%%%%%%%%%%%%%%%%%%%%%%%%%%%%%%%%%%%%%%%%%%%%%%%%%%%%%%%%%%%%%%%%%%

%%%%%%%%%%%%%%%%%%%%%%%%%%%%%%%%%%%%%%%%%%%%%%%%%%%%%%%%%%%%%%%%%%%%%%%%%%%%%%%%%%%

\section{Appendix A: $L^2$-gradients of averaging functions}
\label{AppendixA}

In this appendix, in a quite general set-up, we state and prove a theorem on multiple
roots of characteristic functions applied in the proofs of Theorem~\ref{Theorem5.2} and
Theorem~\ref{Theorem5.3}.

\begin{thm}\label{TheoremA.1}
Let $F, \chi : {\mathbb C} \times L^2_c \rightarrow {\mathbb C}$
be analytic maps and $\psi$ an arbitrary but fixed element in $L^2_c$. Assume that
at $z_\psi \in {\mathbb C}, \ \chi (\cdot , \psi )$ has a zero of order $m \geq 1$. Then
the following statements hold:

(i) For any $\varepsilon > 0$ sufficiently small there exists an open neighborhood
${\mathcal V} \subseteq L^2_c$ of $\psi $ such that for any $\varphi \in {\mathcal V},
\chi (\cdot , \varphi )$ has exactly $m$ roots $z_1(\varphi ), \ldots , z_m(\varphi )$,
listed with their multiplicities, in the open disk $D^\varepsilon \equiv D
^\varepsilon (z_\psi ):= \{ \lambda \in {\mathbb C}\,|\,|\lambda - z_\psi |< \varepsilon \}$
and no roots on the boundary $\partial D^\varepsilon $ of $D^\varepsilon $.

(ii) The functional $F_\chi : {\mathcal V} \rightarrow {\mathbb C}$, defined by
   \[ F_\chi (\varphi ):= \sum ^m_{j=1} F(z_j(\varphi ), \varphi ),
   \]
is analytic and at $(\varphi , \lambda ) = (\psi , z_\psi )$
   \[ \partial F_\chi = m\,\partial F + \sum ^m_{j=0} a_j \partial ^{m-j}_\lambda
      \partial \chi
   \]
where $a_j \in {\mathbb C}, \ 0 \leq j \leq m$, and $\partial $ denotes the $L^2$-gradient
with respect to $\varphi $ and $\partial_\lambda$ denotes the derivative with respect to $\lambda$.
If $F(\cdot , \psi )$ has a zero of order $k \geq 1$ at
$z_\psi $, then $a_0 = \ldots = a_{k-1} = 0$; if $k = m$, then
   \[ a_m = - \frac{1}{m!}\,\partial^m_\lambda \Big(F(\lambda , \psi ) \frac{(\lambda -
      z_\psi )^{m+1} \partial_\lambda\chi (\lambda , \psi )}{\chi (\lambda , \psi )^2}\Big)
       \Big\arrowvert_{\lambda = z_\psi } \not= 0 .
   \]
\end{thm}

\begin{rem}
\label{RemarkA.2} As $\chi : {\mathbb C} \times L^2_c \rightarrow {\mathbb C}$ is analytic
it follows that for any  $\varphi \in L^2_c$,  $\partial \chi : {\mathbb C} \rightarrow L^2_c$,
$\lambda \mapsto \partial\chi\big|_{(\varphi,\lambda)}$ is analytic and so is
$\partial^k_\lambda \partial \chi $ for any $k \geq 1$.
\end{rem}

\begin{proof} (i) By the analyticity of $\chi (\cdot , \psi )$ there exists $\varepsilon> 0$ so that
$\chi (\cdot , \psi )$ does not vanish on $\overline{D^\varepsilon}\setminus\{z_\psi\}$. 
By the analyticity of $\chi $ it then follows that there exists a neighborhood ${\mathcal V}$ of $\psi $
in $L^2_c$ so that for any $\varphi \in {\mathcal V}, \ \chi (\cdot , \varphi )$ does not vanish in
a small tubular neighborhood of $\partial D^\varepsilon$ in $\mathbb C$. It then follows by
the argument principle that for any $\varphi \in {\mathcal V}, \ \chi (\cdot , \varphi )$ has precisely $m$ zeros
in $D^\varepsilon $, when counted with their multiplicities.

(ii) Again by the argument principle, for any $\varphi \in {\mathcal V}$ one has
   \begin{equation}
   \label{A.1} F_\chi (\varphi ) = \frac{1}{2\pi i} \int _{\partial D^\varepsilon }
               F(\lambda , \varphi )\,\frac{\dot \chi (\lambda , \varphi )}
                     {\chi(\lambda , \varphi )}\,d\lambda
   \end{equation}
where $\dot \chi = \partial _{\lambda }\chi $. Note that the integrand in \eqref{A.1}
is analytic on $\partial D^\varepsilon \times {\mathcal V}$, whence $F_\chi $ is
analytic on ${\mathcal V}$. To compute its $L^2$-gradient $\partial F_\chi $ it is
convenient to introduce for $g = (g_1, g_2), h= (h_1, h_2) \in L^2_c$,
   \[ \langle g, h \rangle _r = \int ^1_0 (g_1 h_1 + g_2 h_2) dx .
   \]
Then, by the definition of the $L^2$-gradient, one has at $\varphi = \psi $,
   \begin{align*} \langle \partial F_\chi , h \rangle _r &= \frac{d}{ds} \Big\arrowvert
                     _{s=0} F_\chi (\psi + sh) \\
                  &= \frac{1}{2\pi i} \int _{\partial D^\varepsilon }\frac{d}{ds}\Big|_{s=0} 
                     \left( F(\lambda , \psi + sh) \frac{\dot \chi
                     (\lambda , \psi + sh)}{\chi (\lambda , \psi + sh)} \right) d\lambda .
   \end{align*}
By the product rule one gets at $\varphi = \psi $
\begin{equation*} 
\langle \partial F_\chi , h\rangle _r = 
\frac{1}{2\pi i}\int_{\partial D^\varepsilon} 
\left[\langle \partial F, h \rangle_r\,
\frac{\dot \chi }{\chi} + 
F\cdot\left(\frac{1}{\chi}\,\langle \partial \dot \chi , h \rangle _r -
\frac{1}{\chi ^2}\,\langle \partial \chi , h \rangle _r \dot \chi\right)\right]\,d\lambda .
\end{equation*}
Hence $\partial F_\chi $ is given by
   \begin{equation}\label{A.2} 
              \frac{1}{2\pi i} \int _{\partial D^\varepsilon}\!\!\!
                  \Big( \frac{\dot \chi }
                  {\chi } \partial F + \frac{1}{(\lambda - z_\psi )^m} \cdot \frac{(
                  \lambda - z_\psi) ^m F }{\chi } (\partial \chi )^\cdot
               - \frac{1}{(\lambda - z_\psi )^{m+1}} \cdot \frac{(\lambda - z_\psi)^{m+1}
                  F}{\chi ^2} {\dot\chi}{\partial\chi} \Big)\, d\lambda .
   \end{equation}
Here we used that $\partial\dot\chi=(\partial\chi)^\cdot$ and that
$\partial F, \partial \chi : {\mathbb C} \rightarrow L^2_c$ are analytic and hence in particular,
the maps $\partial F, \partial \chi , (\partial \chi )^\cdot : {\mathbb C} \rightarrow L^2_c$ are continuous.
Furthermore, as by assumption, $\chi (\cdot , \psi )$ has a zero of order $m$ at 
$\lambda = z_\psi$, $\frac{(\lambda - z_\psi)^mF}{\chi}(\partial\chi)^\cdot$ and 
$\frac{(\lambda - z_\psi )^{m+1} F\dot \chi }{\chi ^2}\partial\chi$ are both analytic functions on
$\overline{D^\varepsilon}$ with values in $L^2_c$. Hence by the argument principle, at $\varphi = \psi$,
   \[ \frac{1}{2\pi i} \int _{\partial D^\varepsilon } \frac{\dot \chi }{\chi } \partial
       F d \lambda = m\,\partial F \Big\arrowvert _{\lambda = z_\psi }
   \]
and by Cauchy's integral formula,
   \begin{align*} &\frac{1}{2\pi i} \int _{\partial D^\varepsilon } \frac{1}{(\lambda -
                     z_\psi )^m} \frac{(\lambda - z_\psi )^mF}{\chi }(\partial\chi)^\cdot\,d\lambda \\
                  &= \frac{1}{(m-1)!} \partial ^{m-1}_\lambda \Big\arrowvert _{\lambda =
                     z_\psi } \left( \frac{(\lambda - z_\psi )^m F}{\chi }
                             (\partial\chi)^\cdot\right)
   \end{align*}
and
   \begin{align*} &\frac{1}{2\pi i} \int _{\partial D^\varepsilon } \frac{1}{(\lambda -
                     z_\psi )^{m+1}} \frac{(\lambda - z_\psi )^{m+1}F\dot \chi }{\chi ^2}
                     \partial
                     \chi d\lambda \\
                  &= \frac{1}{m!} \partial ^{m}_{\lambda} \Big\arrowvert _{\lambda =
                     z_\psi } \left( \frac{(\lambda - z_\psi )^{m+1} F\dot \chi }{\chi ^2}
                     \partial \chi \right) .
   \end{align*}
Thus $\partial F_\chi $ at $\varphi = \psi $ is given by
\begin{eqnarray*}
\partial F_\chi &=&m\,\partial F \Big\arrowvert _{\lambda = z_\psi } + \frac{1}{(m-1)!}\,
\partial^{m-1}_\lambda \Big\arrowvert_{\lambda = z _\psi}
\left(\frac{(\lambda - z_\psi)^m F}{\chi}\,\partial _\lambda \partial \chi \right)\\
      &-& \frac{1}{m!}\,\partial ^m_\lambda \Big\arrowvert _{\lambda = z _\psi }
      \left( \frac{(\lambda - z_\psi )^{m+1} F \dot \chi }{\chi ^2}\,\partial \chi \right) .
\end{eqnarray*}
The claimed formula for $\partial F_\chi $ at $\varphi = \psi $ then follows from the
Leibniz rule. If $F(\cdot , \psi )$ has a zero of order $k \geq 1$ at $\lambda = z
_\psi $, then at $(\varphi ,\lambda ) = (\psi , z_\psi )$
   \[ \partial F_\chi = m\,\partial F + \sum ^m_{j=k} a_j \partial ^{m-j}_\lambda
      \partial \chi ,
   \]
i.e., $a_j = 0$ for $0 \leq j \leq k - 1$. If $k = m$, then
   \[ \partial F_\chi = m\,\partial F + a_m \partial \chi
   \]
where in this case
   \[ a_m = - \frac{1}{m!}\,\partial ^m_\lambda \Big(F(\lambda , \psi ) \frac{(\lambda - z
      _\psi )^{m+1} \dot \chi (\lambda , \psi )}{\chi (\lambda , \psi )^2}\Big)
      \Big|_{\lambda = z_\psi } \not= 0 .
   \]
\end{proof}

\medskip

Finally we record a few simple facts from linear algebra, also needed in
Section~\ref{section5}. Consider $f = (f_1, f_2)$ in $L^2_c$ and denote by
$\ell \equiv \ell _f$ the ${\mathbb R}$-linear functional on the ${\mathbb R}$-vector
space $\iLR$ induced by $f$,
   \[ \ell : \iLR \rightarrow {\mathbb C}, h \mapsto \langle f , h \rangle _r ,
   \]
where
   \[ \langle f, h \rangle _r = \int ^1_0 (f_1 h_1 + f_2 h_2) dx .
   \]
Write $\ell (h)$ as $\ell _R(h) + i \ell _I(h)$ where $\ell _R \equiv \ell _{f,R}$
and $\ell _I \equiv \ell _{f,I}$ are the elements in the dual ${\mathcal L}(\iLR,
{\mathbb R})$ of $\iLR$ given by
\begin{equation}\label{A.3}
\ell_R(h) = \re(\langle f,h \rangle _r) \mbox { and } \ell _I(h) =
      \im (\langle f, h\rangle _r)\,.
\end{equation}
They can be expressed in terms of $f$ and $\hat f = - (\overline f_2, \overline f_1)$
as follows
\begin{equation}\label{A.4} 
\ell _R(h) = \Big\langle \frac{f + \hat f}{2}, h\Big\rangle _r \mbox { and }
                 \ell _I(h) = \Big\langle \frac{f - \hat f}{2i}, h\Big\rangle_r .
\end{equation}
As the subspace $\iLR \subseteq L^2_c$ is the subset of all elements $\varphi\in L^2_c$
satisfying $\varphi = \hat\varphi$ it follows that $\frac{f + \hat f}{2}$ and $\frac{f - \hat f}{2i}$
are in $iL^2_r$.

Using that $f \mapsto \hat f$ is an involution and that for any $c$ in ${\mathbb C}$,
$\widehat{(c f)} = \overline c \hat f$, the following lemma can be proved in a straightforward way.

\begin{Lm}
\label{LemmaA.3} (i) $\ell _R, \ell _I$ are ${\mathbb R}$-linearly independent iff
$\frac{f + \hat f}{2}, \frac{f - \hat f}{2i}$ are ${\mathbb R}$-linearly independent.

(ii) $\ell _R, \ell _I$ are ${\mathbb R}$-linearly independent iff
$\frac{f + \hat f}{2}$, $\frac{f - \hat f}{2i}$ are ${\mathbb C}$-linearly independent.

(iii) For any $\lambda \in {\mathbb C}\setminus\{ 0\}$,
$\ell _{f,R}$, $\ell _{f,I}$ are ${\mathbb R}$-linearly independent iff
$\ell_{\lambda f,R}$, $\ell _{\lambda f,I}$ are $\mathbb R$-linearly independent.

(iv) $\ell _{f,R}, \ell _{f,I}$ are ${\mathbb R}$-linearly {\it dependent} iff
there exists $\lambda \in {\mathbb C} \backslash \{ 0 \} $ so that
\[
\frac{\lambda f + \widehat{(\lambda f)}}{2} = 0 .
\]
\end{Lm}

%%%%%%%%%%%%%%%%%%%%%%%%%%%%%%%%%%%%%%%%%%%%%%%%%%%%%%%%%%%%%%%%%%%%%%%%%%%%%%%%%%%%%%%%

%%%%%%%%%%%%%%%%%%%%%%%%%%%%%%%%%%%%%%%%%%%%%%%%%%%%%%%%%%%%%%%%%%%%%%%%%%%%%%%%%%%%%%%%

\section{Appendix B: Examples}
\label{seexamples}

In this section we consider potentials in $\iLR$ of the form $(a \in \mathbb C, k \in \mathbb Z$)
\begin{equation} \label{generic41}
\varphi_{a,k}(x)=(a e^{2\pi i k x}, -\bar a e^{-2 \pi i k x}).
\end{equation}
Most of the results presented in this section can be found in \cite{LiML}. We include them
for the convenience of the reader.
First we show that we can easily relate various spectra of $L(\varphi_{a,k})$ with the
corresponding ones for $k=0$.
More generally, for an arbitrary potential $\varphi \in \Lc$, various spectra of
$L(\varphi_1 e^{2\pi i k x}, \varphi_2e^{-2\pi i k x})$ are related to the corresponding
spectra of $(\varphi_1, \varphi_2)$ by the following lemma which can be verified in a
straightforward way.
\begin{Lm}
Assume that $f=(f_1,f_2)$ is a solution of $L(\varphi)f= \lambda f$ where $\varphi \in \Lc$
is arbitrary. Then
$(f_1 e^{i \pi k  x}, f_2e^{- i \pi k x} )$ is a solution of
\[
L(\varphi_1 e^{2\pi i k x}, \varphi_2e^{-2\pi i k x})g=(\lambda- k \pi) g.
\]
\end{Lm}

\begin{cor} \label{corgen42}
For any $\varphi \in \Lc$ and $k \in \mathbb Z$, the fundamental solution
\[
\check M(x, \lambda) \equiv M(x, \lambda, (\varphi_1 e^{2i\pi  k x}, \varphi_2e^{-2i\pi  k x}))
\]
of $L(\varphi_1 e^{2\pi i k x}, \varphi_2e^{-2\pi i k x})$ is related to the fundamental solution
$M(x, \lambda)$ of $L(\varphi_1, \varphi_2)$ by
\[
\check M= \text{\rm diag }(e^{i \pi k x}, e^{-i \pi k x}) \cdot M(x, \lambda + k \pi).
\]
\end{cor}

Corollary \ref{corgen42} yields the following application.

\begin{pr} \label{genericpro43}
For any $\varphi=(\varphi_1, \varphi_2) \in \Lc$ and any $k \in \mathbb Z$,
\[
\spec_p (L(\varphi_1 e^{2\pi i k x}, \varphi_2e^{-2\pi i k x}))= 
\spec_p(L(\varphi_1, \varphi_2)) - k\pi
\]
and
\[
\spec_D (L(\varphi_1 e^{2\pi i k x}, \varphi_2e^{-2\pi i k x}))=
\spec_D (L(\varphi_1, \varphi_2)) - k\pi
\]
(with multiplicities).
\end{pr}

\begin{proof}
Recall that the characteristic functions $\chi_{p}$ and $\chi_{D}$ are
given by
\begin{align*}
\chi_{p}(\lambda)&=(\grave{m}_1(\lambda)+\grave{m}_4(\lambda))^2-4 \\
2 i \chi_{D}(\lambda)&=
\grave{m}_4(\lambda)+\grave{m}_3(\lambda)-\grave{m}_2(\lambda)-\grave{m}_1(\lambda).
\end{align*}
By Corollary \ref{corgen42},
\[
\chi_{p}(\lambda, (\varphi_1 e^{2\pi i k x}, \varphi_2e^{-2\pi i
k x}))=\chi_{p}(\lambda + k\pi, \varphi)
\]
and
\[
\chi_{D}(\lambda, (\varphi_1 e^{2\pi i k x}, \varphi_2e^{-2\pi i
k x}))=(-1)^{k}\chi_{D}(\lambda + k\pi, \varphi).
\]
As $\spec_p (L(\varphi_1 e^{2\pi i k x}, \varphi_2e^{-2\pi i k x}))$
and $\spec_D (L(\varphi_1 e^{2\pi i k x}, \varphi_2e^{-2\pi i k x}))$
are the zero sets (with multiplicities) of
$\chi_{p}(\lambda, (\varphi_1 e^{2\pi i k x}, \varphi_2e^{-2\pi i k x}))$
respectively $\chi_{D}(\lambda, (\varphi_1 e^{2\pi i k x}, \varphi_2e^{-2\pi i k x}))$,
the claimed identities follow.
\end{proof}

In view of Proposition \ref{genericpro43}, instead of the potentials $\varphi_{a,k}$
defined by (\ref{generic41}),
it suffices to consider the case $k=0$,
\[
\varphi_a \equiv \varphi_{a,0}=(a, - \bar a), \quad a \in \mathbb C.
\]
In a straightforward way one verifies the following

\begin{Lm}\label{genericlm44} 
For any $a \in \mathbb C$,
\begin{equation} \label{genericform42}
M(x, \lambda, \varphi_a)= \left( \begin{array}{cc} \cos( \kappa x)- i \lambda\,
\frac{ \sin(\kappa x) }{ \kappa } &  ia\,\frac{\sin(\kappa x)}{\kappa} \\  i \bar a\,
\frac{\sin(\kappa x)}{\kappa} & \cos( \kappa x)+ i \lambda\,\frac{\sin(\kappa x)}{\kappa}
\end{array} \right)
\end{equation}
where
\begin{equation}\label{eq:kappa*}
\kappa\equiv \kappa(\lambda, a)= \sqrt{\lambda^2+|a|^2}
\end{equation}
\end{Lm}

\begin{rem}
Note that $\kappa$ depends only on the modulus $|a|$ of $a$ and that the right hand side of \eqref{genericform42}
does not depend on the choice of the sign of the root $\sqrt{\lambda^2+|a|^2}$ as cosine is an even
function whereas sine is odd. Furthermore, the right hand side of \eqref{genericform42} is well defined
at $\kappa=0$ as $\frac{\sin(\kappa x)}{ \kappa }=x+O(\kappa^2)$\,.
\end{rem}

\medskip

\noindent{\it Periodic spectrum of $L(\varphi_a)$}: By Lemma \ref{genericlm44} one has
$\Delta(\lambda, \varphi_a)=2\cos \kappa(\lambda)$ and hence the characteristic
function of $L(\varphi_a)$, considered on the interval $[0,2]$ with periodic boundary
conditions, is given by
\begin{equation} \label{genericform43}
\chi_{p}(\lambda, \varphi_a)=\Delta^2(\lambda, \varphi_a)-4=-4\sin^2(\kappa(\lambda)).
\end{equation}
The periodic eigenvalues of $L(\varphi_a)$ are thus given by the $\lambda$'s
satisfying $\kappa(\lambda)=n \pi$ for some
$n \in \mathbb Z$, or
\begin{equation} \label{genericform44}
\lambda^2+ |a|^2= n^2 \pi^2.
\end{equation}
The monodromy matrix $\grave{M}$ for such a $\lambda$ is given by
\begin{equation}\label{eq:monodromy1}
\grave{M}=\left( \begin{array}{cc}(-1)^n & 0 \\ 0 & (-1)^n \end{array} \right)
\end{equation}
when $n\ne 0$ and by
\begin{equation}\label{eq:monodromy2}
\grave{M} =  \left( \begin{array}{cc} 1-i\lambda &i a  \\
i \bar a  &1+i\lambda \end{array} \right)
= \left( \begin{array}{cc} 1\pm |a| &i a  \\  i \bar a & 1 \mp |a| \end{array} \right)
\end{equation}
when $n=0$. It is convenient to list the periodic eigenvalues {\it not} in lexicographic ordering,
but rather use the integer $n \in \mathbb Z$ in (\ref{genericform44}) as an index. When
listed in this way, we denote the periodic eigenvalues by
$\hat \lambda_n ^\pm$, $n \in \mathbb Z$, which are defined as follows.
For any $n \in \mathbb Z$ with $|n \pi| > |a|$ denote
\[
\hat \lambda_n^+=\hat\lambda_n^-= \text{\rm sgn}(n)\cdot\sqrt[+]{n^2 \pi^2 - |a|^2}.
\]
In view of \eqref{eq:monodromy1}, $\hat \lambda_n^+$ defined above is a periodic eigenvalue of
$L(\varphi_a)$ of geometric multiplicity two.
Using \eqref{eq:kappa*} and \eqref{genericform43} one easily sees that
$\hat\lambda^+_n$ has algebraic multiplicity is two. Further, for any $n \in \mathbb Z$ with 
$0<|n \pi| < |a|$ denote
\[
\hat\lambda_n^+=\hat\lambda_n^-= \text{\rm sgn}(n)\cdot i\sqrt[+]{|a|^2 -n^2\pi^2}.
\]
Again, in view of \eqref{eq:monodromy1}, for $n \in {\mathbb Z}$ with $0 < |n\pi | < |a|$,
$\hat \lambda_n^+$ is a periodic eigenvalue of $L(\varphi_a)$ of geometric multiplicity two
and, by \eqref{eq:kappa*} and \eqref{genericform43}, its algebraic multiplicity is two. 
Next note that for $n = 0$, one has $\hat \lambda ^\pm_0 = \pm i|a|$. In view of \eqref{eq:monodromy2},
for $a \not= 0$ the geometric multiplicity of $\hat \lambda_0^+$ as well as
of $\hat \lambda_0^-$ equals one. In view of \eqref{eq:kappa*} and \eqref{genericform43} the algebraic
multiplicity of $\hat \lambda_0^+$ and the one of $\hat \lambda_0^-$ is one. For $a \not= 0$,
the eigenfunctions corresponding to $\hat \lambda_0^+$ and  $\hat \lambda_0^-$ are the constant vectors
$\left( a, i|a| \right)$ resp. $\left( a, -i|a| \right)$.  We then obtain the following result,
used in the proof of Theorem~\ref{Theorem5.2}.

\begin{Lm}\label{Lemma8.5} For any $k \in {\mathbb Z}$, consider the potential
$\varphi _{a,-k} = (ae^{-2 i\pi k x}, -\overline ae^{2i \pi k x})$.
Then $\hat \lambda ^\pm _0 = k \pi \pm i|a|$ are periodic eigenvalues of $L(\varphi _{a,-k})$ of 
algebraic multiplicity one.
\end{Lm}

In the special case where $|a|=n_a \pi$ for some $n_a \in \mathbb Z_{>0}$ set $\hat
\lambda^{\pm}_{\pm n_a}=0$.
The above computations yield
\begin{cor}\label{genericcor45}
\begin{itemize}
\item[(i)] For $a \in \mathbb C$, $\varphi_a$ is a standard potential iff $|a|< \pi$.
\item[(ii)] For $a\in\mathbb C$, any multiple periodic eigenvalue $\lambda$ of $L(\varphi_a)$
satisfies $m_p(\lambda)=2$ and $m_g(\lambda)=2$ iff $|a|>\pi$ and $|a|\ne\pi\mathbb Z$.
\item[(iii)] If $a \in \mathbb C \setminus \left\{0 \right\} $ satisfies $|a| \in \pi
\mathbb Z$, then $0$ is a periodic eigenvalue of $L(\varphi_a)$ of algebraic multiplicity four.
\end{itemize}
\end{cor}

\medskip

\noindent{\it Isospectral set $\text{Iso}_0(\varphi_a)$: } Denote by $\text{Iso}_0(\varphi_a)$
the connected component containing $\varphi_a$ of the set $\text{Iso}(\varphi_a)$ of
all potentials $\varphi \in \iLR$ with $\spec_pL(\varphi)=\spec_pL(\varphi_a)$.
By the computations above one sees that
\[
\{|a| e^{i \alpha} \, |\, \alpha \in \mathbb R \} \subseteq \text{Iso}_0(\varphi_a).
\]
For $|a|$ sufficiently small, $\varphi_a$ is in the domain of the Birkhoff map introduced in
Theorem 1.1 in \cite{KLTZ}. As the $L_2$-norm is a spectral invariance it then follows that, for $|a|$ sufficiently
small, all of $\text{Iso}(\varphi_a)$ is contained in this domain. According to the computations above $\varphi_a$ is
a $1$-gap potential. It then follows from Theorem 1.1 in \cite{KLTZ} and its proof that $\text{Iso}(\varphi_a)$ is
homeomorphic to a circle. As a consequence
\[
\text{Iso}(\varphi_a)=\text{Iso}_0(\varphi_a) =\{|a| e^{i \alpha} \, |\, \alpha \in \mathbb R \}.
\]
Most likely the latter identities remain true for any $|a|< \pi$, but we have not verified this. For
$|a|>\pi$,  Li and McLaughlin observed that $\text{Iso}_0(\varphi_a)$ is larger than
$\{|a| e^{i \alpha} \, |\, \alpha \in \mathbb R \}$. Indeed, let $\pi < |a| < 2 \pi$. Then
$\hat \lambda^+_{\pm 1}= \pm i \sqrt[+]{|a|^2-\pi^2}$ are periodic eigenvalues of geometric multiplicity two.
In subsection 4.3 of \cite{LiML}, using
B\"acklund transformation techniques, formulas of solutions of fNLS are presented which evolve on
$\text{Iso}_0(\varphi_a)$ and depend explicitly on $x$. They are parametrized by the punctured complex plane 
$\mathbb C^*:=\left\{e^{\rho}e^{i\beta}\right\}$ with coordinates
$(\rho, \beta) \in \mathbb R \times \mathbb R / 2 \pi \mathbb Z$, whereas the angle variable $\alpha $ in 
$\{|a| e^{i \alpha} \, |\, \alpha \in \mathbb R \}$ is proportional to the time $t$. 
As $t\rightarrow \pm \infty$ these solutions approach the $x$ independent solutions evolving on 
$\{|a| e^{i \alpha} \, |\, \alpha \in \mathbb R \}$. Due to the trace formulas
(\cite{LiML}, Section 2.4), on the orbits of these solutions, the periodic eigenvalues $\hat{\lambda}^+_{\pm 1}$
have geometric multiplicity one.

\medskip

\noindent{\it Dirichlet spectrum of $L(\varphi_a)$:} 
By Lemma \ref{genericlm44}, the characteristic function of the Dirichlet spectrum of $L(\varphi_a)$ is given by
\begin{equation} \label{genericform49}
\chi_{D}(\lambda, \varphi_a)=\frac{\sin \kappa}{\kappa}\left( \lambda + \frac{\bar a - a}{2} \right).
\end{equation}
The Dirichlet eigenvalues of $L(\varphi_a)$ are thus given by the $\lambda$'s satisfying
\begin{equation} \label{genericform410}
\kappa(\lambda)= n \pi
\end{equation}
for some  $n \in \mathbb Z\setminus \left\{0\right\}$ or
\begin{equation} \label{genericform411}
 \lambda + \frac{\bar a - a}{2}=0.
\end{equation}
Note that by the definition of the Dirichlet boundary conditions, any Dirichlet eigenvalue is of
geometric multiplicity one.
It is convenient to list the Dirichlet eigenvalues {\it not} in lexicographic ordering, but rather use the integer
$n$ in (\ref{genericform410}) as an index. When listed in this way, we denote the Dirichlet eigenvalues by
$\hat \mu_n$, $n \in \mathbb Z$, which are defined as follows. For all $n \in \mathbb Z$ with $|n\pi|>|a|$ denote
\[
\hat \mu_n= \text{\rm sgn}(n)\cdot\sqrt[+]{n^2 \pi^2-|a|^2}.
\]
From \eqref{genericform49} it follows that $\hat \mu_n$ has algebraic multiplicity one. 
For all $n \in \mathbb Z$ with $0<|n \pi | < |a|$ denote
\[
\hat \mu_n= \text{\rm sgn}(n) \cdot i \sqrt[+]{|a|^2-n^2 \pi^2}.
\]
By the same arguments as in the case $|n\pi|> |a|$, the algebraic multiplicity of $\hat \mu_n$ is equal to one iff
$\hat \mu_n+ \frac{\bar a - a}{2} \not = 0$ and two otherwise. For $n=0$ denote
\[
\hat\mu_0= i \im(a).
\]
Again, by the same arguments, $\hat \mu_0$ has algebraic multiplicity equal to one if
\[
[\im(a) \not=0 \text{ and } \im(a)\not = \pm \sqrt[+]{|a|^2- n^2 \pi ^2}\,\, \forall \, 0 < |n \pi|< |a|]
\text{ or } [ \im(a)=0 \text{ and } |a| \not \in \pi \mathbb Z_{>0}]
\]
or two if
\[
\im(a) \in \left\{\pm \sqrt[+]{|a|^2- n^2 \pi ^2}\,|\,0< |n \pi|< |a|\right\}
\]
or three if
\[
\im(a) =0 \quad \text{ and } \quad |a| \in \pi \mathbb Z_{>0}.
\]
In the special case where $|a|= n_a\pi$ for some $n_a \in \mathbb Z_{>0}$ one has $\hat \mu_{n_a}=\hat \mu_{-n_a}=0$.
The algebraic multiplicity of $\hat \mu_{n_a}$ is two ($\im(a) \not= 0$) or three ($\im(a) =0$).
These computations lead to the following
\begin{cor} \label{genericcor47}
Let $a \in \mathbb C$. Then
\begin{itemize}
\item[(i)] If $|a|< \pi$, then the Dirichlet spectrum of $L(\varphi_a)$ is simple.
\item[(ii)] If $|a| \not\in  \pi \mathbb Z_{>0}$, then the only possible multiple Dirichlet eigenvalue is $i\im(a)$.
It is at most of algebraic multiplicity two.
\item[(iii)] If $|a| \in  \pi \mathbb Z_{>0}$, then $0$ is a Dirichlet eigenvalue of algebraic multiplicity two
or three.
\item[(iv)] For any $0 < n \pi < |a| $ or $n \pi > |a|$, $\hat{\mu}_n$ is a periodic eigenvalue
of geometric and algebraic multiplicity two whereas for $|a|= n \pi \in \pi \mathbb Z_{> 0}$, 
$\hat{\mu}_n=0$ is a periodic eigenvalue of algebraic multiplicity $4$.
\end{itemize}
\end{cor}

\end{document}